\documentclass[a4paper,11pt]{article}
\usepackage{graphicx}
\usepackage{amssymb}
\usepackage{amsthm}
\usepackage{amsfonts}
\usepackage{amsmath}
\usepackage{hyperref}

\usepackage{xcolor}
\title{Constructing orientable and negative orientable sequences with asymptotically optimal period}
\author{Chris J. Mitchell and Peter R. Wild}
\date{14th August 2026}
\theoremstyle{plain}
\newtheorem{lemma}{Lemma}[section]
\newtheorem{theorem}[lemma]{Theorem}
\newtheorem{corollary}[lemma]{Corollary}

\theoremstyle{definition}
\newtheorem{definition}{Definition}[section]

\newtheorem{example}{Example}[section]

\theoremstyle{remark}
\newtheorem{remark}{Remark}[section]

\begin{document}

\maketitle

\begin{abstract}
Orientable sequences, periodic sequences in which any $n$-tuple appears at most once in either
direction, were introduced in the early 1990s for use in certain position location applications;
constructions and upper bounds on the period for the binary case were published by Dai et al. More
recent work has focussed on $k$-ary sequences for arbitrary $k>2$; one method of construction
involves negative orientable sequences, in which an $n$-tuple appears at most once in either the
sequence or the negative of its reverse.  In this paper we show how additional $n$-tuples can be
added to one previously described approach for generating negative orientable sequences, resulting
in new sequences with asymptotically optimal period.  These sequences can in turn be used to
generate orientable sequences, again with asymptotically optimal period.

\end{abstract}

\section{Introduction} \label{section:introduction}

\subsection{Background}  \label{subsection:background}

Orientable sequences were introduced in the early 1990s \cite{Burns92,Burns93,Dai93} in the context
of a position sensing application.  An orientable sequence of order $n$ is an infinite periodic
sequence with symbols drawn from a finite alphabet --- typically $\mathbb{Z}_k$ for some $k$ ---
with the property that any particular subsequence of length $n$, referred to throughout as an
$n$-tuple, occurs at most once in a period \emph{in either direction}.  That is, if anyone reading
the sequence observes $n$ consecutive symbols, they can deduce both the direction in which they are
reading and their position within one period of the sequence. Gabri\'{c} and Sawada \cite{Gabric25}
provide an interesting discussion of further possible applications as well as their relationship to
strings relevant to DNA computing. Note that for the purposes of this article the
\emph{period} of a sequence $(a_i)$ is the least positive integer $n$ such that $a_{i+n}=a_i$ for
every $i$.

The early work referred to above focussed on the binary case, i.e.\ where $k=2$; Dai et al.\
\cite{Dai93} provided both an upper bound on the period for this case and a method of construction
yielding sequences with asymptotically optimal periods. More recently, Mitchell and Wild
\cite{Mitchell22} showed how the Lempel homomorphism \cite{Lempel70} could be applied to
recursively generate binary orientable sequences with periods a large fraction of the optimal
value.  In 2024, Gabri\'{c} and Sawada \cite{Gabric24,Gabric24b} described a highly efficient
method of generating binary orientable sequences with the largest known periods.

In 2024, Alhakim et al.\ \cite{Alhakim24a} studied the general alphabet case, i.e.\ where $k>2$.
They gave an upper bound on the period of orientable sequences for all $n$ and $k$, and also
described a range of methods of construction using the Alhakim and Akinwande generalisation of the
Lempel homomorphism to arbitrary finite alphabets \cite{Alhakim11}.  Since then a range of
construction methods have been proposed \cite{Gabric25,Mitchell26,Mitchell25a}; of particular
interest is the method of Gabri\'{c} and Sawada \cite{Gabric25}, who showed how to construct
sequences with asymptotically optimal period for any $n$ and $k>2$ using a cycle-joining approach.
The construction method described in this paper is thus not the first to yield sequences with
asymptotically optimal period, but for the values of $n$ and $k$ we investigated the periods of the
generated sequences are greater in every case.

\subsection{Motivation and structure}  \label{subsection:motivation}

The notion of the \emph{pseudoweight} of a $k$-ary $n$-tuple is central to this paper. The
pseudoweight of an $n$-tuple is simply the integer sum of the elements in the tuple, added to $k/2$
times the number of zeros, i.e.\ the weight of the tuple if every 0 is replaced by $k/2$. The
pseudoweight of any $n$-tuple clearly lies between $n$ and $n(k-1)$.

A second concept fundamental here is the set $E_k(n-1)$ of $k$-ary $n$-tuples with pseudoweight
less than the `middle' value $nk/2$.  There are $\frac{k^n-r_{k,n,kn/2}}{2}$ such tuples, where
$r_{k,n,s}$ is the number of $k$-ary $n$-tuples with pseudoweight $s$. In an earlier paper
\cite{Mitchell26} we showed how to construct a family of \emph{negative orientable sequences},
where a negative orientable sequence is one in which an $n$-tuple occurs only once in the period of
a sequence or the negative of its reverse, where the negative is computed modulo $k$.  The period
of such a sequence is $|E_k(n-1)|=\frac{k^n-r_{k,n,kn/2}}{2}$, which approaches the maximum
possible; this sequence includes precisely the tuples in $E_k(n-1)$. This in turn enables us to
construct orientable sequences with large period. The key properties making the construction
possible are that: (a) $E_k(n-1)$ possesses the antinegasymmetry property, i.e.\ the reverse of the
negative of any $n$-tuple in $E_k(n-1)$ is not in $E_k(n-1)$ --- this is simple to see since the
negative of a $n$-tuple with pseudoweight $s$ has pseudoweight $nk/2-s$, and (b) $E_k(n-1)$ is
Eulerian i.e.\ it is the set of edges of an Eulerian subgraph of the de Bruijn digraph $B_k(n-1)$
(defined in Subsection~\ref{subsection:notation} below).

The goal of this paper is to consider how we might add certain $n$-tuples of pseudoweight $kn/2$ to
$E_k(n-1)$ whilst preserving its antinegasymmetry and its Eulerian property.  This enables us to
construct negative orientable sequences of period greater than
$\frac{k^n-r_{k,n,kn/2}}{2}$, and indeed of asymptotically optimal period.

The remainder of the paper is organised as follows.  Following this introductory section, in
Section~\ref{section:circuit_partition} an approach to dividing all the $n$-tuples of pseudoweight
exactly $nk/2$ into circuits is described; this partition is of fundamental importance to the rest
of the paper. Section~\ref{section:adding_tuples} then describes how the $n$-tuples in certain of
these circuits can be added to the set $E_k(n-1)$ of $k$-ary $n$-tuples with pseudoweight less than
$nk/2$ without losing the key properties enabling the construction of a negative orientable
sequence. This leads to Section~\ref{section:circuit_enumeration}, in which a lower bound is
developed on the number of $n$-tuples that can be added to $E_k(n-1)$. The final main part of the
paper, Section~\ref{section:OSs}, shows how new orientable and negative orientable sequences can be
constructed from the enlarged version of $E_k(n-1)$, and lower bounds on the sequence lengths are
obtained; it is also shown that the sequences have asymptotically optimal period. The paper
concludes in Section~\ref{section:conclusions} with a brief discussion of possible future work.

\subsection{Notation and basic results}  \label{subsection:notation}

We next formalise some of the key concepts we introduced in the introductory text. In this paper we
consider periodic sequences $(s_i)$ with elements from $\mathbb{Z}_k$ for some $k$, which we refer
to as $k$-ary. Since we are interested in tuples occurring either forwards or backwards in a
sequence, we also introduce the notion of a reversed tuple, so that if $\mathbf{u} =
(u_0,u_1,\ldots,u_{n-1})$ is a $k$-ary $n$-tuple then $\mathbf{u}^R = (u_{n-1},u_{n-2},
\ldots,u_0)$ is its \emph{reverse}.  We are also interested in negating all the elements of a
tuple, and hence if $\mathbf{u} = (u_0,u_1,\ldots,u_{n-1})$ is a $k$-ary $n$-tuple, we write
$-\mathbf{u}$ for $(-u_0,-u_1,\ldots,-u_{n-1})$. We write $\mathbf{s}_n(i)$ for the tuple
$(s_i,s_{i+1},\ldots,s_{i+n-1})$.

\begin{definition}[\cite{Alhakim24a}]
A $k$-ary \emph{$n$-window sequence $S = (s_i)$} (also known as a \emph{cut-down de Bruijn
sequence} --- see, for example, \cite{cameron25}) is a periodic sequence of elements from
$\mathbb{Z}_k$ ($k>1$, $n>1$) with the property that no $n$-tuple appears more than once in a
period of the sequence, i.e.\ with the property that if $\mathbf{s}_n(i) = \mathbf{s}_n(j)$ for
some $i,j$, then $i \equiv j \pmod m$ where $m$ is the period of the sequence.
\end{definition}

This paper is concerned with two special classes of $n$-window sequences: orientable and negative
orientable sequences.

\begin{definition}[\cite{Alhakim24a}] A $k$-ary $n$-window sequence $S = (s_i)$ is
said to be an \emph{orientable sequence of order $n$}, an $\mathcal{OS}_k(n)$, if $\mathbf{s}_n(i)
\neq \mathbf{s}_n(j)^R$, for any $i,j$.
\end{definition}

\begin{definition}[\cite{Alhakim24a}]  \label{definition:NOS}
A $k$-ary $n$-window sequence $S=(s_i)$ is said to be a \emph{negative orientable sequence of order
$n$}, a $\mathcal{NOS}_k(n)$, if $\mathbf{s}_n(i)\not=-{\mathbf{s}_n(j)}^R$, for any $i,j$.
\end{definition}

As described above, the notion of the pseudoweight of a $k$-ary $n$-tuple, as defined
in \cite{Mitchell25a}, is central to this note --- we write $w^*$ for the pseudoweight function.
Using the notation of \cite{Mitchell25a}, we also write $r_{k,n,s}$ for the number of $k$-ary
$n$-tuples with pseudoweight exactly $s$, where $n\leq s\leq n(k-1)$.

Let $B_k(n-1)$ be the de Bruijn digraph with vertices labeled with $k$-ary $(n-1)$-tuples and edges
labeled with $k$-ary $n$-tuples. Following \cite{Mitchell26}, let $E_k(n-1)$ be the set of edges in
$B_k(n-1)$ with pseudoweight less than $kn/2$.  Then, from \cite[Theorem 5.3]{Mitchell26}, the
subgraph of $B_k(n-1)$ with edge set $E_k(n-1)$ and vertices those which have an outgoing or
incoming edge in $E_k(n-1)$ contains $\frac{k^n-r_{k,n,kn/2}}{2}$ edges and every vertex has
in-degree equal to its out-degree; moreover, by the argument in \cite[Section 3.3]{Mitchell25a}
this subgraph is connected, and hence it is Eulerian.

A $k$-ary $n$-tuple $\mathbf{a}$ is said to be negasymmetric if $\mathbf{a}=-\mathbf{a}^R$. We
extend this terminology to $B_k(n-1)$ and refer to an edge or vertex being negasymmetric if its
associated $n$-tuple or $(n-1)$-tuple respectively is negasymmetric. A negasymmetric edge must have
pseudoweight $kn/2$. Such an edge cannot belong to an $\mathcal{NOS}_k(n)$. We are interested in
those edges of pseudoweight $kn/2$ which are not negasymmetric and the possibility that a
proportion of them might be joined to $E_k(n-1)$ that preserves the Eulerian property and hence can
be used to yield an $\mathcal{NOS}_k(n)$ of larger period.  The following elementary results will
be of use later.

\begin{lemma}  \label{lemma:number_of_nega_tuples}
Suppose $k\geq 3$ and $n\geq3$.  The number of negasymmetric $n$-tuples is
\[
\begin{split}
2k^{\lfloor n/2\rfloor} & \text{~~~~if $k$ is even and $n$ is odd, and} \\
k^{\lfloor n/2\rfloor} & \text{~~~~otherwise.}
\end{split}
\]
\end{lemma}

\begin{proof}
The first $\lfloor n/2\rfloor$ elements of a negasymmetric $n$-tuple determine the last $\lfloor
n/2\rfloor$ elements.  Moreover, if $n$ is odd then the middle element of the tuple must be $0$ or
$k/2$. The result follows.
\end{proof}

\begin{lemma}  \label{lemma:nega_moved_by_halfn_is_nega}
Suppose $(a_0,a_1,\dots,a_{n-1})$ is a negasymmetric $n$-tuple.  Then the following $n$-tuples are
also negasymmetric:
\[
\begin{split}
(a_{n/2},a_{n/2+1},\dots,a_{n-1},a_0,a_1,\dots,a_{n/2-1}) & \text{~~~~if $n$ is even, and} \\
(a_{(n+1)/2},a_{(n+1)/2+1},\dots,a_{n-1},x,a_0,a_1,\dots,a_{(n-3)/2}) & \text{~~~~if $n$ is odd.}
\end{split}
\]
where $x$ is either 0 or $k/2$.
\end{lemma}

\begin{proof}
The result follows immediately from the definition of negasymmetric.
\end{proof}

Again following \cite{Mitchell26}, a subgraph of $B_k(n-1)$ is said to be antinegasymmetric if
$\mathbf{a}\neq-\mathbf{b}^R$ for every pair of edges, $(\mathbf{a},\mathbf{b})$, in this subgraph.
The subgraph of $B_k(n-1)$ with edge set $E_k(n-1)$ is clearly antinegasymmetric, and since it is
also Eulerian there exists a Eulerian circuit in this subgraph corresponding to a sequence $S$
which is an $\mathcal{NOS}_k(n)$ of period $\frac{k^n-r_{k,n,kn/2}}{2}$ (see \cite[Lemma
3.15]{Mitchell25a}). The reverse of $-S$ covers the edges of pseudoweight greater than $kn/2$ (as,
of course, does $-S$).

Let $H_k(n-1)$ be the subgraph of $B_k(n-1)$ restricted to edges of pseudoweight precisely $kn/2$.
The nature of $H_k(n-1)$ will be quite different depending on whether
$n$ is odd or even and hence we consider the cases
separately. Observe that, by definition, there
are $r_{k,n,kn/2}$ edges in $H_k(n-1)$.

\section{Partitioning \texorpdfstring{$H_k(n-1)$}{Hk(n-1)} into necklaces}  \label{section:circuit_partition}

We first need the following concept that is key to the ideas in this paper.

\begin{definition}[Necklace --- see, for example, \cite{Ruskey92}]
A $k$-ary \emph{necklace} is a cycle of period dividing $n$ within the de Bruijn graph $B_k(n-1)$
that includes all edges ($n$-tuples) that are cyclic shifts of one another.  If
$(a_0,a_1,\dots,a_{n-1})$ is a $k$-ary $n$-tuple, i.e.\ an edge in $B_k(n-1)$, then let $m$ be the
smallest positive integer $c$ such that $a_i=a_{\overline{i+c}}$ for every $i$ ($0\leq i<n$), where
$\overline{x}=x\bmod n$. We write $[a_0,a_1,\dots,a_{n-1}]$ for the necklace in $B_k(n-1)$
consisting of edges
\[ (a_0,a_1,\dots,a_{n-1}), (a_1,a_2\dots,a_{n-1},a_0), \dots,
(a_{m-1},\dots,a_{n-1},a_0,a_1,\dots,a_{m-2}), \] and say the necklace has period $m$.
\end{definition}

We start by showing how the edges of $H_k(n-1)$ can readily be divided into edge-disjoint
necklaces.

\subsection{Defining the partition}  \label{subsection:defining_Ckn-1}

We first establish the following.

\begin{lemma}  \label{lemma:k_odd_vertex_degrees}  \label{lemma:Hkn-1_vertex_degrees}
Suppose $k\geq3$ and $n\geq 3$. Suppose $\mathbf{a}=(a_1,a_2\dots,a_{n-1})$ is a vertex of
$H_k(n-1)$. Then:
\begin{itemize}
\item if $k$ is odd then $\mathbf{a}$ has in-degree and out-degree:
\begin{align*}
1 & \text{~~if~~} w^*(\mathbf{a})\in\{kn/2-i:1\leq i\leq k-1\} \text{~or~} w^*(\mathbf{a})=k(n-1)/2 \\
0 & \text{~~otherwise}.
\end{align*}
\item if $k$ is even then $\mathbf{a}$ has in-degree and out-degree:
\begin{align*}
2 & \text{~~if~~} w^*(\mathbf{a})=k(n-1)/2 \\
1 & \text{~~if~~} w^*(\mathbf{a})\in\{kn/2-i:1\leq i\leq k-1\} \\
0 & \text{~~otherwise}.
\end{align*}
\end{itemize}
\end{lemma}

\begin{proof}
If $(a_0,a_1,a_2\dots,a_{n-1})$ is an edge incoming to $\mathbf{a}$ then, since this edge must have
    pseudoweight $kn/2$, it follows by definition that

\[ \sum_{i=0}^{n-1}w^*(a_i) = kn/2, \]
i.e.\
\[ \sum_{i=1}^{n-1}w^*(a_i) = kn/2 - w^*(a_0). \]

\begin{itemize}
\item Suppose $k$ is odd. Since $w^*(a_0)\in\{i:1\leq i\leq k-1\}$ or $w^*(a_0)=k/2$, the
    result follows given that $w^*(a_0)=k/2$ if and only if $a_0=0$.
\item Suppose $k$ is even. Exactly the same argument applies except that $w^*(a_0)=k/2$ if and
    only if $a_0=0$ or $a_0=k/2$.
\end{itemize}
\end{proof}

We next present a simple way of dividing the edges ($n$-tuples) in $H_k(n-1)$ into edge-disjoint
necklaces.

\begin{lemma}  \label{lemma:k_odd_partitioning edges}  \label{lemma:defining_circuits}
Suppose $k\geq3$ and $n\geq 3$.
\begin{enumerate}
\item[i)] If $(a_0,a_1,\dots,a_{n-1})$ is a $k$-ary $n$-tuple of pseudoweight $kn/2$, i.e.\ an
    edge in $H_k(n-1)$, then $[a_0,a_1,\dots,a_{n-1}]$ is a necklace in $H_k(n-1)$ of
     period dividing $n$.
\item[ii)] Let~$\mathcal{C}_k(n-1)$ be the set of all necklaces of the form
    $[a_0,a_1,\dots,a_{n-1}]$, where $(a_0,a_1,\dots,a_{n-1})$ is an edge in $H_k(n-1)$; then
    $\mathcal{C}_k(n-1)$ forms an edge-disjoint partition of the edges in $H_k(n-1)$.
\end{enumerate}
\end{lemma}

\begin{proof}
\begin{enumerate}
\item[i)] Suppose $(a_0,a_1,\dots,a_{n-1})$ is any edge in $H_k(n-1)$.  Then every cyclic shift
    of $(a_0,a_1,\dots,a_{n-1})$ is also an edge in $H_k(n-1)$, since cyclic shifting does not
    alter pseudoweight.  That is, $[a_0,a_1,\dots,a_{n-1}]$ is a necklace in $H_k(n-1)$ with
    period that divides $n$.
\item[ii)] Any edge $(a_0,a_1,\dots,a_{n-1})$ is in a unique necklace of this type, and hence
    these necklaces cover all possible edges and are edge-disjoint.
\end{enumerate}
\end{proof}

\begin{corollary}
If $k\geq3$ is odd and $n\geq3$ then the necklaces in $\mathcal{C}_k(n-1)$ form a unique vertex-
and edge-disjoint partition of $H_k(n-1)$.
\end{corollary}

\begin{proof}
From Lemma~\ref{lemma:k_odd_vertex_degrees}, if $k$ is odd, every vertex in $H_k(n-1)$ has
in-degree and out-degree at most one, and hence the necklaces in $\mathcal{C}_k(n-1)$ are
vertex-disjoint. Finally, the fact that this partition is unique follows immediately from the fact
that the in-degree (and out-degree) of every vertex in $H_k(n-1)$ is at most one.
\end{proof}

\begin{remark}  \label{remark:other_partitions}
Whilst the partition $\mathcal{C}_k(n-1)$ is a unique partition of the edges in $H_k(n-1)$ when $k$
is odd, this is not the case for $k$ even; neither is the partition vertex-disjoint in the $k$ even
case. That is, there may be other partitions for the $k$ even case which give better results than
those we achieve below using $\mathcal{C}_k(n-1)$. This remains an area
for future work.
\end{remark}

The above lemma motivates the following extension of the notion of negasymmetry.

\begin{definition}
A necklace in $H_k(n-1)$ is said to be \emph{negasymmetric} if there are edges, $\mathbf{a}$ and
$\mathbf{b}$ say (not necessarily distinct), in this necklace such that $\mathbf{a}=-\mathbf{b}^R$.
A necklace is \emph{not} negasymmetric, i.e.\ \emph{non-negasymmetric}, only if the subgraph of
$B_k(n-1)$ it defines is antinegasymmetric.
\end{definition}

Of course, if the necklace contains a negasymmetric edge then clearly the circuit is negasymmetric.

We have the following simple technical results.

\begin{lemma}  \label{lemma:nega_for_c}
Suppose $n\geq3$, $k\geq3$, and $c\mid n$.  Then:
\begin{itemize}
\item[(i)] Suppose $(a_0,a_1,\dots,a_{n-1})$ is a $k$-ary $n$-tuple with $a_{i}=a_{i+c}$ for
    every $i$ ($0\leq i\leq n-c-1$).  Then $(a_0,a_1,\dots,a_{n-1})$ is a negasymmetric
    $n$-tuple if and only if $(a_0,a_1,\dots,a_{c-1})$ is a negasymmetric $c$-tuple. Moreover,
    in this case if $c$ is odd then $a_{(c-1)/2}=0$ or $k/2$ (the latter only applying if $k$
    is even).
\item[(ii)] Suppose $[a_0,a_1,\dots,a_{n-1}]$ is a necklace in $B_k(n-1)$ of period dividing
    $c$. Then $[a_0,a_1,\dots,a_{n-1}]$ is a negasymmetric necklace in $B_k(n-1)$ if and only
    if $[a_0,a_1,\dots,a_{c-1}]$ is a negasymmetric necklace in $B_k(c-1)$. Moreover,
    $[a_0,a_1,\dots,a_{n-1}]$ is a negasymmetric necklace in $H_k(n-1)$ if and only if
    $[a_0,a_1,\dots,a_{c-1}]$ is a negasymmetric necklace in $H_k(c-1)$.
\end{itemize}
\end{lemma}

\begin{proof}
\begin{itemize}

\item[(i)] If $(a_0,a_1,\dots,a_{n-1})$ is negasymmetric, then $a_i=-a_{n-1-i}$ for every $i$
    ($0\leq i\leq n-1$).  But, since $a_{i}=a_{i+c}$ for every $i$ and $c\mid n$,
    $a_{n-1-i}=a_{c-1-i}$ for every $i$, and hence $(a_0,a_1,\dots,a_{c-1})$ is negasymmetric.
    If $(a_0,a_1,\dots,a_{c-1})$ is negasymmetric, then $a_i=-a_{c-1-i}$ for every $i$ ($0\leq
    i\leq c-1$).  But $a_{tc+i}=a_{i}$ for every $t$ ($0\leq t\leq n/c-1$), and hence
    $a_i=-a_{n-1-i}$ for every $i$, and so $(a_0,a_1,\dots,a_{n-1})$ is negasymmetric. Finally,
    if $(a_0,a_1,\dots,a_{c-1})$ is negasymmetric and $c$ is odd then, by definition,
    $a_{(c-1)/2}=-a_{(c-1)/2}$ and the result follows.

\item[(ii)] If $[a_0,a_1,\dots,a_{n-1}]$ is negasymmetric then $a_i=-a_{\overline{t-i}}$ for
    some $t$ and every $i$ ($0\leq i\leq n-1$), where the bar indicates reduction modulo $n$.
    Hence $a_i=-a_{\overline{\overline{t-i}}}$, where the double bar indicates reduction modulo
    $c$, since $c\mid n$.  Hence $[a_0,a_1,\dots,a_{c-1}]$ is negasymmetric. If
    $[a_0,a_1\dots,a_{c-1}]$ is negasymmetric, then $a_i=-a_{\overline{\overline{t-i}}}$ for
    some $t$ and every $i$ ($0\leq i\leq c-1$). But $a_{tc+i}=a_{i}$ for every $t$ ($0\leq
    t\leq n/c-1$), and hence $a_i=-a_{\overline{t-i}}$ for every $i$ ($0\leq i\leq n-1$), and
    thus $[a_0,a_1\dots,a_{n-1}]$ is negasymmetric. Finally, it is immediate that
    $[a_0,a_1,\dots,a_{n-1}]$ has pseudoweight $kn/2$ if and only if $[a_0,a_1,\dots,a_{c-1}]$
    has pseudoweight $kc/2$.
\end{itemize}
\end{proof}

\begin{lemma}  \label{lemma:multiple_negasymmetric_tuples}
Suppose $n\geq3$, $k\geq3$ and $[a_0,a_1,\dots,a_{n-1}]$ is a necklace in $B_k(n-1)$ containing two
negasymmetric $n$-tuples at distance $t$ apart; without loss of generality suppose $0<t\leq n/2$.
Then $[a_0,a_1\dots,a_{n-1}]$ has period dividing $2t$.
\end{lemma}

\begin{proof}
Without loss of generality suppose the two negasymmetric $n$-tuples are
\[ (a_0,a_1,\dots,a_{n-1}) \text{~~and~~} (a_{\overline{t}},a_{\overline{t+1}}\dots,a_{\overline{t-1}}) \]
where $\overline{s}=s\bmod n$.  Then, by definition
\[ a_{\overline{i}}=-a_{\overline{-1-i}} \text{~~and~~} a_{\overline{t+i}}=-a_{\overline{t-1-i}} \]
for every $i$ ($0\leq i<n$).  Setting $j=i+t$ in the second equation gives
\[ a_{\overline{j}}=-a_{\overline{2t-1-j}} \]
for every $j$ ($0\leq j<n$).  Hence
\[ a_{\overline{-1-i}} = a_{\overline{2t-1-i}} \]
for every $i$ ($0\leq i<n$), and hence
\[ a_{\overline{i}} = a_{\overline{2t+i}} \]
for every $i$ ($0\leq i<n$). The desired result follows.
\end{proof}

\begin{corollary}  \label{corollary:two_negatuples_means_n_even}
Suppose $n\geq3$, $k\geq3$, and that $[a_0,a_1,\dots,a_{n-1}]$ has
period $c$ and contains two distinct negasymmetric $n$-tuples. Then $c$ is even and the two
negasymmetric $n$-tuples must be at distance exactly $c/2$ apart.
\end{corollary}

\begin{proof}
Suppose the two negasymmetric $n$-tuples are at distance $t$ apart, where, without loss of
generality, $0< t\leq c/2$. Since the two tuples are distinct, $c$ cannot be a factor of $t$;
however, by Lemma~\ref{lemma:multiple_negasymmetric_tuples}, $c\mid 2t$, and hence $c$ must be
even. Since $c\mid 2t$ and $t\leq c/2$ we have $t=c/2$, and the result follows.
\end{proof}

\subsection{Examples}  \label{subsection:examples}

To illustrate the above results we next give some simple examples.

\begin{example}  \label{example:n3k3}
Suppose $n=k=3$. First observe that, from \cite[Table 2]{Mitchell25a},
$r_{k,n,kn/2}=r_{3,3,4.5}=7$. The seven edges in $H_3(2)$ can be partitioned into three necklaces:
\[ [000], [012], [021] \]
of periods 1, 3 and 3, respectively. All three necklaces are negasymmetric and contain one
negasymmetric $3$-tuple, accounting for the three negasymmetric 3-tuples (see
Lemma~\ref{lemma:number_of_nega_tuples}).
\end{example}

\begin{example}  \label{example:n3k5}
Suppose $n=3$ and $k=5$. From \cite[Table 2]{Mitchell25a}, $r_{k,n,kn/2}=r_{5,3,7.5}=13$. The 13
edges in $H_5(2)$ can be partitioned into five necklaces:
\[ [000], [014], [041], [023], [032] \]
all of which have period 3 except the first. All five necklaces are negasymmetric and contain one
negasymmetric $3$-tuple, accounting for the five negasymmetric 3-tuples (see
Lemma~\ref{lemma:number_of_nega_tuples}).
\end{example}

\begin{example}  \label{example:n4k3}
Suppose $n=4$ and $k=3$. From \cite[Table 2]{Mitchell25a}, $r_{k,n,kn/2}=r_{3,4,6}=19$. The 19
edges in $H_3(3)$ can be partitioned into six necklaces:
\[ [0000], [1212], [0012], [0021], [0102], [1122] \]
all of which have period 4 except the first and second (with periods 1 and 2 respectively). All six
necklaces are negasymmetric.  [0102] contains no negasymmetric 4-tuples, [0000] contains one
negasymmetric 4-tuple, and the other four necklaces contain two negasymmetric 4-tuples.  We write
$N_0=1, N_1=1, N_2=4$ to denote this. This accounts for the nine negasymmetric 4-tuples (see
Lemma~\ref{lemma:number_of_nega_tuples}).
\end{example}

\begin{example}  \label{example:n5k3}
Suppose $n=5$ and $k=3$. From \cite[Table 2]{Mitchell25a}, $r_{k,n,kn/2}=r_{3,5,7.5}=51$. The 51
edges in $H_3(4)$ can be partitioned into 11 necklaces, all but the first of which have period 5:
\[
\begin{split}
 & [00000], [00012], [00021], [00102], [00201], 01122], [01212], \\
 & [02211], [02121], [01221], [02112].
\end{split}
\]
The first nine of these necklaces are negasymmetric and contain one negasymmetric $3$-tuple, but
the last two, i.e.\ [01221] and [02112], are not negasymmetric. That is, as we discuss below, one
of these two necklaces could be joined to $S$ to result in a $\mathcal{NOS}_3(5)$ of period five
greater than the period given in \cite[Table 3]{Mitchell25a}, i.e.\ 96+5=101.  The set of 101
5-tuples making up $S$ can, in turn, be used to construct an $\mathcal{OS}_3(6)$ of period
$k\times101=303$ (greater than the previous record of 288 and close to the best known upper bound
of 315).
\end{example}

\begin{example}  \label{example:n6k3}
Suppose $n=6$ and $k=3$. From \cite[Table 2]{Mitchell25a}, $r_{k,n,kn/2}=r_{3,6,9}=141$.  The 141
edges in $H_3(5)$ can be partitioned into 26 necklaces: one of period 1, one of period 2, two of
period 3, and 22 of period 6.  Eight of these 26 necklaces, all of period 6, are non-negasymmetric,
namely:
\[ [001221], [002112], [010122], [020211], [010212], [020121], [010221], [020112]. \]
A further three necklaces of period 6 are negasymmetric, but do not contain any negasymmetric
6-tuples, namely:
\[ [000102], [000201], [011022]. \]
Three of the remaining necklaces, namely those with periods 1 and 3, are negasymmetric and contain
a single negasymmetric 6-tuple:
\[ [000000], ~~~~[012012], [021021]. \]
Observe that the two necklaces of period 3 yield the same sequences  as the negasymmetric necklaces
$[012]$ and $[021]$ in $B_3(2)$.  Finally, the remaining 12 necklaces, one of period 2 and eleven
of period 6, contain two negasymmetric 6-tuples:
\[
\begin{split}
 & [121212], [000012], [000021], [001002], [001122], 002211], \\
 & [001212], [002121], [012021], [111222], [112212], [221121].
\end{split}
\]
Using the previous notation, this means that $N_0=3$, $N_1=3$ and $N_2=12$; this
accounts for the 27 negasymmetric 6-tuples (see
Lemma~\ref{lemma:number_of_nega_tuples}).

So, as discussed below, the 6-tuples in four of the eight non-negasymmetric 6-tuples, i.e.\ a total
of 24 6-tuples, could be joined to $S$ to result in a $\mathcal{NOS}_3(6)$ of period 24 greater
than the period given in \cite[Table 3]{Mitchell25a}, i.e.\ 294+24=318. The set of 318 6-tuples
making up $S$ can, in turn, be used to construct an $\mathcal{OS}_3(7)$ of period $k\times318=954$
(greater than the previous record of 882 and close to the best known upper bound of 972).
\end{example}

\begin{example}  \label{example:n3k4}
Suppose $n=3$ and $k=4$. From \cite[Table 2]{Mitchell25a}, $r_{k,n,kn/2}=r_{4,3,6}=20$.  The 20
edges in $H_4(2)$ can be partitioned into 8 necklaces: two of period 1 and six of period 3, all of
which are negasymmetric.  All the eight necklaces contain a single negasymmetric 3-tuple,
accounting for the eight negasymmetric 3-tuples (see Lemma~\ref{lemma:number_of_nega_tuples}).
\end{example}

\begin{example}  \label{example:n4k4}
Suppose $n=4$ and $k=4$. From \cite[Table 2]{Mitchell25a}, $r_{k,n,kn/2}=r_{4,4,8}=70$.  The 70
edges in $H_4(3)$ can be partitioned into 20 necklaces: two of period 1, two of period 2 and 16 of
period 4. Four of the necklaces, all of period 4, are non-negasymmetric, namely:
\[ [0132], [0213], [0231], [0312]. \]

A further seven necklaces, six of of period 4 and one of period 2 are negasymmetric, but do not
contain any negasymmetric 4-tuples, namely:
\[ [0002], [0222], [0103], [0123], [0321], [1232],~~~~[0202]. \]

Two of the remaining necklaces, namely those with period 1, are negasymmetric and contain a single
negasymmetric 4-tuple:
\[ [0000], [2222]. \]

Finally, the remaining seven necklaces, one of period 2 and six of period 4, each contain two
negasymmetric 4-tuples:
\[ [1313],~~~~[0022], [0013], [1003], [2213], [2231], [1133]. \]
This means that $N_0=7$, $N_1=2$ and $N_2=7$, accounting for the 16 negasymmetric
4-tuples (see Lemma~\ref{lemma:number_of_nega_tuples}).

So, the 4-tuples in two of the four non-negasymmetric 4-tuples, i.e.\ a total of eight 4-tuples,
could be joined to $S$ to result in a $\mathcal{NOS}_4(4)$ of period 8 greater than the period
given in \cite[Table 3]{Mitchell25a}, i.e.\ 93+8=101. The set of 101 4-tuples making up $S$ can, in
turn, be used to construct an $\mathcal{OS}_4(5)$ of period $k\times 101=404$.
\end{example}

In the remainder of the paper we generalise the observations made in the examples to
obtain formulae for the number of $n$-tuples that can be added to $E_k(n-1)$ whilst preserving the
key properties that enable us to construct a negative orientable sequence containing precisely
these tuples.

\subsection{Properties of necklaces in \texorpdfstring{$\mathcal{C}_k(n-1)$}{Ck(n-1)}}

As mentioned previously, the goal of this paper is to
consider how we might add certain $n$-tuples of pseudoweight $kn/2$ to $E_k(n-1)$ whilst preserving
its antinegasymmetry and its Eulerian property. This would enable us to construct negative
orientable sequences of period greater than $\frac{k^n-r_{k,n,kn/2}}{2}$.

To achieve this goal we need to consider the properties of the necklaces in $\mathcal{C}_k(n-1)$.
We first have the following key result for odd $k$.

\begin{lemma}  \label{lemma:complete_ciruits_odd_k}
Suppose $k\geq3$ is odd and $n\geq 3$. Suppose $X$ is an Eulerian and antinegasymmetric subset of
the edges of $B_k(n-1)$ containing $E_k(n-1)$ as a subset. Then $X-E_k(n-1)$ only contains edges
corresponding to complete necklaces from $\mathcal{C}_k(n-1)$, and these necklaces cannot be
negasymmetric.
\end{lemma}

\begin{proof}
First observe that $X$ cannot contain any edges of pseudoweight greater than $kn/2$, and hence
$X-E_k(n-1)\subseteq H_k(n-1)$.  Since both $X$ and $E_k(n-1)$ are Eulerian, this means that every
vertex of the subgraph induced by $X-E_k(n-1)$ has in-degree equal to its out-degree.  Since no
vertex in $H_k(n-1)$ has in-degree or out-degree greater than one, the fact that $X-E_k(n-1)$ is
made up of necklaces follows from Lemma~\ref{lemma:k_odd_partitioning edges}. Finally, no necklace
in $X-E_k(n-1)$ can be negasymmetric, since that would contradict the requirement for $X$ to be
antinegasymmetric.
\end{proof}

\begin{remark}
It seems unlikely that the above result is true for $k$ even.
\end{remark}

In fact we can say more about the set of necklaces making up $\mathcal{C}_k(n-1)$, regardless of
whether $k$ and $n$ are odd or even.

\begin{lemma}  \label{lemma:circuit_pairs_k_odd}
Suppose $k\geq3$ and $n\geq 3$. If $[a_0,a_1,\dots,a_{n-1}]$ is a necklace in $\mathcal{C}_k(n-1)$
then $[-a_0,-a_1,\dots,-a_{n-1}]$ is also a necklace in $\mathcal{C}_k(n-1)$. Moreover, if
$[a_0,a_1,\dots,a_{n-1}]$ is non-negasymmetric then the reverse-complementary necklace
$[-a_{n-1},-a_{n-2},\dots,-a_{0}]$ is also non-negasymmetric and $[a_0,a_1,\dots,a_{n-1}]$ shares
no edges with $[-a_{n-1},-a_{n-2},\dots,-a_{0}]$.
\end{lemma}

\begin{proof}
Since all the edges in $[a_0,a_1,\dots,a_{n-1}]$ have pseudoweight $nk/2$ then so do all the edges
in $[-a_0,-a_1,\dots,-a_{n-1}]$.  Hence $[-a_0,-a_1,\dots,-a_{n-1}]$ is a necklace in
$\mathcal{C}_k(n-1)$.

Next suppose $[a_0,a_1,\dots,a_{n-1}]$ is non-negasymmetric and
$[a_0,a_1,\dots,a_{n-1}]=[-a_{n-1},-a_{n-2},\dots,-a_{0}]$, where equality here means that they are
the same necklace.  Hence
$(a_0,a_1,\dots,a_{n-1})=(-a_{\overline{t}},-a_{\overline{t-1}},\dots,-a_{\overline{t+1}})$ for
some $t$, where here, as throughout, $\bar{u}$ denotes the unique integer satisfying $0\leq
\overline{u}\leq n-1$ and $u\equiv\overline{u}\pmod n$. But this means that
$[a_0,a_1,\dots,a_{n-1}]$ is negasymmetric, giving a contradiction.

Finally, if $[-a_{n-1},-a_{n-2},\dots,-a_{0}]$ is negasymmetric, then
\[ (-a_{n-1},-a_{n-2},\dots,-a_{0})=(a_{\overline{t}},a_{\overline{t+1}},\dots,a_{\overline{t-1}}) \]
for some $t$, and hence
\[ (a_{0},a_{1},\dots,a_{n-1})=(-a_{\overline{t-1}},-a_{\overline{t-2}},\dots,-a_{\overline{t}}),
\]
which again contradicts the assumption that the necklace $[a_0,a_1,\dots,a_{n-1}]$ is
non-negasymmetric.
\end{proof}

\begin{remark}
It follows from Lemma~\ref{lemma:circuit_pairs_k_odd} that the non-negasymmetric necklaces in
$\mathcal{C}_k(n-1)$ can be divided into pairs consisting of a necklace and the reverse of its
negative.  We refer to these as \emph{reverse-complementary pairs}.

More formally, define a mapping $I: H_k(n-1)\rightarrow H_k(n-1)$ such that
$(a_0,a_1,\dots,a_{n-1})\rightarrow(-a_{n-1},-a_{n-2},\dots,-a_{0})$, i.e. $I$ maps an $n$-tuple to
its reverse-negative. This induces a involution on the set of necklaces $\mathcal{C}_k(n-1)$ where
$[a_0,a_1,\dots,a_{n-1}]\rightarrow[-a_{n-1},-a_{n-2},\dots,-a_{0}]$. Necklaces fixed by $I$ are
negasymmetric, and the other orbits under $I$ are of size 2 (making up the reverse-complementary
pairs).
\end{remark}

\begin{example}  \label{example:revisit}
Revisiting Example~\ref{example:n5k3}, the two non-negasymmetric necklaces [01221] and [02112] form
a reverse-complementary pair.  Similarly, revisiting Example~\ref{example:n6k3}, the eight
non-negasymmetric necklaces can be divided into four reverse-complementary pairs, as follows:
\[
\begin{split}
([001221], [002112]);~~ ([010122], [020112]);\\
([010212], [020121]);~~ ([010221], [020211]).
\end{split}
\]
Similarly, looking at Example~\ref{example:n4k4}, the four non-negasymmetric necklaces can be
divided into two reverse-complementary pairs, as follows:
\[ ([0132], [0213]),~~([0231],[0312]). \]

\end{example}

Because of Lemma~\ref{lemma:complete_ciruits_odd_k}, it is thus of interest to know how many
reverse-complementary pairs of non-negasymmetric necklaces there are in $\mathcal{C}_k(n-1)$, and
what their periods are.  A first step in this direction is the following result.

\begin{lemma}  \label{lemma:nega_ciruit_means_nega_tuple_odd_k_odd_c}
Suppose $k\geq3$ and $n\geq 3$.  Suppose $[a_0,a_1,\dots,a_{n-1}]$ is a negasymmetric necklace in
$\mathcal{C}_k(n-1)$ of period $c$, where $c$ is odd ($c\mid n$).  Then
$(a_{\overline{s}},a_{\overline{s+1}},\dots,a_{\overline{s-1}})$ is a negasymmetric $n$-tuple for
some $s$ ($0\leq s\leq n-1$).  Moreover, $[a_0,a_1,\dots,a_{n-1}]$ contains precisely one
negasymmetric $n$-tuple.
\end{lemma}

\begin{proof}
Since $[a_0,a_1,\dots,a_{n-1}]$ is negasymmetric, by definition we have
\[ (a_0,a_1,\dots,a_{n-1}) = (-a_{\overline{t}},-a_{\overline{t-1}},\dots,-a_{\overline{t+1}}) \]
for some $t$.  Hence
\[ a_{\overline{i}}=-a_{\overline{t-i}} \]
for every $i$, $0\leq i<n$.

If $t$ is odd, setting $i=(t+1)/2+j$ gives
\[ a_{\overline{(t+1)/2+j}}=-a_{\overline{(t+1)/2-1-j}}, \]
for every $j$, $0\leq j<n$.  Hence
\[ (a_{\overline{s}},a_{\overline{s+1}},\dots,a_{\overline{s-1}}) =
(-a_{\overline{s-1}},-a_{\overline{s-2}},\dots,-a_{\overline{s}}) \]
where $s=(t+1)/2$, and so $(a_{\overline{s}},a_{\overline{s+1}},\dots,a_{\overline{s-1}})$ is
negasymmetric.

If $t$ is even, setting $i=(t-c+1)/2+j$ (which is an integer since $c$ is odd) gives

\[ a_{\overline{(t-c+1)/2+j}}=-a_{\overline{t-(t-c+1)/2-j}}=
-a_{\overline{(t-c+1)/2+c-1-j}}= -a_{\overline{(t-c+1)/2-1-j}}\]

for every $j$, $0\leq j<n$, since $[a_0,a_1,\dots,a_{n-1}]$ has period $c$. Thus
\[ (a_{\overline{s}},a_{\overline{s+1}},\dots,a_{\overline{s-1}}) =
(-a_{\overline{s-1}},-a_{\overline{s-2}},\dots,-a_{\overline{s}}) \] where $s=(t-c+1)/2$, and so
$(a_{\overline{s}},a_{\overline{s+1}},\dots,a_{\overline{s-1}})$ is negasymmetric.  That is
$[a_0,a_1,\dots,a_{n-1}]$ contains at least one negasymmetric $n$-tuple.

The fact that the $n$-tuple is unique follows immediately from
Corollary~\ref{corollary:two_negatuples_means_n_even}.
\end{proof}

\begin{remark}  \label{remark:n_even_counterexample}  \label{remark:n_odd_OK}
Since $c\mid n$, Lemma~\ref{lemma:nega_ciruit_means_nega_tuple_odd_k_odd_c} immediately implies
that if $n$ is odd then all the negasymmetric necklaces in $\mathcal{C}_k(n-1)$ contain a unique
negasymmetric $n$-tuple. However, the above result does not hold if $c$ is even, and as a result we
examine the $n$ odd and even cases separately.

To see why the $c$ even case is different, suppose
\[ (a_0,a_1,\dots,a_{n-1}) = (-a_{\overline{t}},-a_{\overline{t-1}},\dots,-a_{\overline{t+1}}) \]
for some $t$.  If $t$ is odd then the same argument as in the above proof shows that there is a
negasymmetric $n$-tuple in the necklace $[a_0,a_1,\dots,a_{n-1}]$.  However, if $t$ is even then
the argument used above does not work, since $c$ being odd is key.  An example in the case $n=4$
and $k=3$ involves the necklace $[a_0=1,a_1=0,a_2=2,a_3=0]$.  It is simple to see that:
\[ (a_0,a_1,a_2,a_3)=(-a_2,-a_1,-a_0,-a_3) \]
i.e.\ the necklace is negasymmetric (in the above terminology $t=2$).  However, none of the four
4-tuples in the necklace are negasymmetric.
\end{remark}

\section{Adding necklaces from \texorpdfstring{$\mathcal{C}_k(n-1)$}{Ck(n-1)} to \texorpdfstring{$E_k(n-1)$}{Ek(n-1)}}  \label{section:adding_tuples}

As mentioned previously, the goal of this paper is to investigate how we can add additional
$n$-tuples to the set $E_k(n-1)$ while maintaining the property that the corresponding subgraph of
$B_k(n-1)$ is both Eulerian and antinegasymmetric. We can now give our main result.

\begin{theorem}  \label{theorem:Xkn-1_is_Eulerian}
Suppose $k\geq3$ and $n\geq 3$. Suppose $X_k(n-1)$ is the subset of edges of $B_k(n-1)$ made up of
the union of $E_k(n-1)$ and the sets of edges from one of each reverse-complementary pair of the
non-negasymmetric necklaces in $\mathcal{C}_k(n-1)$. Then:

\begin{itemize}
\item[(i)] the subgraph of $B_k(n-1)$ with edge set $X_k(n-1)$ and vertices with non-zero
    in-degree is Eulerian;
\item[(ii)] the subgraph of $B_k(n-1)$ with edge set $X_k(n-1)$ is antinegasymmetric.
\end{itemize}
\end{theorem}

\begin{proof}
\begin{itemize}
\item[(i)] Since the subgraph of $B_k(n-1)$ with edge set $E_k(n-1)$ is Eulerian (as
    established in \cite[Section 3.3]{Mitchell25a}) and we are adjoining complete necklaces, it
    is clear that the subgraph of $B_k(n-1)$ with edge set $X_k(n-1)$ has in-degree equal to
    out-degree for every vertex. It remains to show that it is connected. We know that the
    subgraph with edge-set $E_k(n-1)$ is connected, and hence it remains to show that every
    added necklace from $H_k(n-1)$ has a vertex in common with one of the edges in $E_k(n-1)$.
    Since $E_k(n-1)$ contains all edges of pseudoweight less than $kn/2$ and $k>2$, all
    vertices in $B_k(n-1)$ with pseudoweight at most $(kn-3)/2$ have at least one outgoing (and
    incoming) edge in $E_k(n-1)$ --- this follows since if $(v_0,v_1,\dots,v_{n-2})$ is such a
    vertex, then $(v_0,v_1,\dots,v_{n-2},1)$ is an outgoing edge in $B_k(n-1)$ of pseudoweight
    at most $(kn-3)/2+1<kn/2$, i.e.\ it is an edge in $E_k(n-1)$. Every added necklace contains
    only edges of pseudoweight $kn/2$, and moreover cannot be the all-zero or all-$kn/2$
    necklace, since they are negasymmetric. That is an added necklace must contain at least one
    element not equal to 1, i.e.\ such a necklace will pass through a vertex with pseudoweight
    less than $kn/2-1$, i.e.\ a vertex with an outgoing edge in $E_k(n-1)$.
\item[(ii)] $E_k(n-1)$ is negasymmetric, as are all the added necklaces.  If $\mathbf{a}$ is an
    edge in an added necklace it must have pseudoweight $kn/2$, and hence $-\mathbf{a}^R$ must
    also have pseudoweight $kn/2$; hence $-\mathbf{a}^R$ is not in $E_k(n-1)$.  Finally,
    suppose that $\mathbf{a}$ and $\mathbf{b}$ are edges in distinct added necklaces.  Then
    $-\mathbf{a}^R$ is in the reverse-complementary necklace to the one containing
    $\mathbf{a}$, and hence $-\mathbf{a}^R\neq\mathbf{b}$.
\end{itemize}
\end{proof}

\begin{example} Reviewing Examples~\ref{example:n6k3} and \ref{example:revisit}, in
the case $n=6$ and $k=3$ the set $X_3(5)$ will contain $E_3(5)$ and four of the eight
non-negasymmetric necklaces in $\mathcal{C}_3(5)$, one of each reverse-complementary pair.  That is
$|X_3(5)|=|E_3(5)|+24=318$, yielding a $\mathcal{NOS}_3(6)$ of period 318.
\end{example}

It therefore remains to determine the size of the set $X_k(n-1)$, as defined in
Theorem~\ref{theorem:Xkn-1_is_Eulerian}, and this is the main focus of the remainder
of the paper.

\section{Enumerating non-negasymmetric necklaces}  \label{section:circuit_enumeration}  \label{section:circuit_enumeration_kodd}

\subsection{Preliminaries}  \label{subsection:preliminaries}

For the reason given in Remark~\ref{remark:n_even_counterexample}, we divide the analysis into two
cases, depending whether $n$ is odd or even. However, we first give a key definition and a lemma
which apply to all $n$.

\begin{definition}
Suppose $k\geq3$ and $n\geq3$. For $i\geq0$, let $N_i(n,k)$ represent the set of negasymmetric
necklaces in $\mathcal{C}_k(n-1)$ containing $i$ negasymmetric $n$-tuples.
\end{definition}

\begin{remark}  \label{remark:Nink at most 2}
If $n$ is odd, it follows immediately from
Lemma~\ref{lemma:nega_ciruit_means_nega_tuple_odd_k_odd_c} that $|N_i(n,k)|=0$ if $i\neq 1$. Also,
from Corollary~\ref{corollary:two_negatuples_means_n_even}, if $n$ is even then a negasymmetric
necklace $[a_0,a_1,\ldots,a_{n-1}]$ contains zero, one or two negasymmetric $n$-tuples, i.e.\
$|N_i(n,k)|=0$ if $i>2$.
\end{remark}

The following result is elementary.

\begin{lemma}\label{lemma:bound on w n evenx}
Suppose $k\geq3$ and $n\geq3$.  Let $n=2^tm$, where $t\geq 0$ and $m$ is odd.  Then the number of
$n$-tuples (edges) in $H_k(n-1)$ that are not in a negasymmetric necklace in $\mathcal{C}_k(n-1)$,
and hence are in a non-negasymmetric necklace in $\mathcal{C}_k(n-1)$, is at least
\[ r_{k,n,kn/2} - n(|N_0(n,k)|+|N_2(n,k)|)-m(|N_1(n,k)|-\delta)-\delta \]
where $\delta=1$ or 2 depending on whether $k$ is odd or even, with equality if $n$ is prime.
\end{lemma}

\begin{proof}
Every negasymmetric necklace in $\mathcal{C}_k(n-1)$ contains at most $n$ $n$-tuples.  However, the
necklaces in $N_1(n,k)$ have odd period, i.e.\ period at most $m$, and hence they contain at most
$m$ $n$-tuples. Moreover, one of the necklaces in $N_1(n,k)$ is the all-zero necklace containing a
single $n$-tuple; if $k$ is even, $N_1(n,k)$ also includes the all-$k/2$ necklace which again
contains a single $n$-tuple. That is, the set of all negasymmetric necklaces contains at most
\[ n|N_0(n,k)| + m(|N_1(n,k)|-\delta) + \delta + n|N_2(n,k)| \]
$n$-tuples.  Now $|H_k(n-1)|=r_{k,n,kn/2}$ and the inequality follows.

Finally, suppose $n$ is prime, and hence odd since $n\geq3$, and so $|N_0(n,k)|=|N_2(n,k)|=0$. The
necklaces in $N_1(n,k)$ must all have `full' period $n$, apart from the all-zero and the all-$k/2$
necklaces, and so all but $\delta$ of these necklaces contain $n$ distinct $n$-tuples. Thus in this
case equality holds.
\end{proof}

The corollary below follows immediately from Remark~\ref{remark:Nink at most 2}.

\begin{corollary}  \label{corollary:edge_bound_nodd}
Suppose $k\geq3$ and $n\geq3$ is odd.  Then the number of $n$-tuples (edges) in $H_k(n-1)$ that are
not in a negasymmetric necklace in $\mathcal{C}_k(n-1)$, and hence are in a non-negasymmetric
necklace in $\mathcal{C}_k(n-1)$, is at least
\[ r_{k,n,kn/2} - n(|N_1(n,k)|-\delta)-\delta \]
where $\delta=1$ or 2 depending on whether $k$ is odd or even, with equality if $n$ is prime.
\end{corollary}

It thus remains to evaluate $|N_i(n,k)|$.

\subsection{The \texorpdfstring{$n$}{n} odd case}  \label{section:circuit_enumeration_kodd_nodd}  \label{section:circuit_enumeration_nodd}

We first examine the $n$ odd case, the much simpler of the two. We have the following
result.

\begin{lemma} \label{lemma:N1_for_n_odd}
Suppose $k\geq3$ and $n\geq3$, where $n$ is odd. Then:
\[ |N_1(n,k)| = \delta k^{(n-1)/2}.\]
\end{lemma}

\begin{proof}
Since $n$ is odd, by Lemma~\ref{lemma:nega_ciruit_means_nega_tuple_odd_k_odd_c} every negasymmetric
necklace in $\mathcal{C}_k(n-1)$ contains a unique negasymmetric $n$-tuple, and so $|N_1(n,k)|$ is
equal to the number of negasymmetric $n$-tuples, and the result follows from
Lemma~\ref{lemma:number_of_nega_tuples}.
\end{proof}

\subsection{The \texorpdfstring{$n$}{n} even case}  \label{section:circuit_enumeration_kodd_neven}

As noted above, if $[a_0,a_1,\dots,a_{n-1}]$ is a negasymmetric necklace of even period in
$\mathcal{C}_k(n-1)$ then the claim of Lemma~\ref{lemma:nega_ciruit_means_nega_tuple_odd_k_odd_c}
does not hold, meaning that it is not so simple to enumerate the negasymmetric necklaces in
$\mathcal{C}_k(n-1)$ when $n$ is even. To obtain a result corresponding to
Lemma~\ref{lemma:N1_for_n_odd} we need to analyse further the behaviour of necklaces in
$\mathcal{C}_k(n-1)$ when $n$ is even.

We first characterise negasymmetric necklaces in $\mathcal{C}_k(n-1)$ with even period that contain
at least one negasymmetric $n$-tuple.  We have the following result, analogous to
Lemma~\ref{lemma:nega_ciruit_means_nega_tuple_odd_k_odd_c}.

\begin{lemma}  \label{lemma:nega_ciruit_means_02_nega_tuples_odd_k_even_c}
Suppose $k\geq3$ and $n\geq 4$ is even.  Suppose $[a_0,a_1,\dots,a_{n-1}]$ is a negasymmetric
necklace in $\mathcal{C}_k(n-1)$ with even period $c$ containing at least one negasymmetric
$n$-tuple.  Then $[a_0,a_1,\dots,a_{n-1}]$ contains precisely two negasymmetric $n$-tuples at
distance $c/2$ apart.
\end{lemma}

\begin{proof}
Suppose $(a_{\overline{s}},a_{\overline{s+1}},\dots,a_{\overline{s-1}})$ is negasymmetric, where
$\overline{x}=x\bmod n$. Then, if $m=n/2$, it follows immediately that
$(a_{\overline{s+m}},a_{\overline{s+m+1}},\dots,a_{\overline{s+m-1}})$ is also negasymmetric, since
$\overline{s+m+i}=\overline{s-m+i}$ for every $i$.  The
result then follows immediately from Corollary~\ref{corollary:two_negatuples_means_n_even}.
\end{proof}

We also have the following result covering the case where a negasymmetric necklace does not contain
any negasymmetric $n$-tuples.

\begin{lemma}  \label{lemma:nega_circuit_no_nega_tuples}
Suppose $k\geq3$, $n\geq 3$, and $[a_0,a_1,\dots,a_{n-1}]$ is a negasymmetric necklace in
$\mathcal{C}_k(n-1)$ containing no negasymmetric $n$-tuples.  Then $n$ is even,
$[a_0,a_1,\dots,a_{n-1}]$ has  period $c$ for some even $c\mid n$, and, for some $s$ ($0\leq s<
c/2$):
\begin{itemize}
\item[(i)] $a_{s}=0$ or $k/2$ and $a_{s+c/2}=0$ or $k/2$;
\item[(ii)] both
\[ (a_{s+1},a_{s+2},\dots,a_{c-1},a_0,a_1,\dots,a_{s-1}) \]
and
\[ (a_{s+c/2+1},a_{s+c/2+2},\dots,a_{c-1},a_0,a_1,\dots,a_{s+c/2-1}) \]
are negasymmetric $(c-1)$-tuples;
\item[(iii)] $[a_0,a_1,\dots,a_{c-1}]$ contains no other negasymmetric
    $(c-1)$-tuples;~moreover, the
    $(n-1)$-tuples of (ii) are distinct;
\item[(iv)]
\[ (a_{s+1},a_{s+2},\dots,a_{n-1},a_0,a_1,\dots,a_{s-1}) \]
and
\[ (a_{s+c/2+1},a_{s+c/2+2},\dots,a_{n-1},a_0,a_1,\dots,a_{s+c/2-1}) \]
are distinct negasymmetric $(n-1)$-tuples,~and
 $[a_0,a_1,\dots,a_{n-1}]$ contains no other negasymmetric $(n-1)$-tuples;~also, if~
 $a_s=a_{s+c/2}$
\[ (a_{s+n/2+1},a_{s+n/2+2},\dots,a_{n-1},a_0,a_1,\dots,a_{s+n/2-1}) \]
 equals either
 \[ (a_{s+1},a_{s+2},\dots,a_{n-1},a_0,a_1,\dots,a_{s-1}) \]
or
\[ (a_{s+c/2+1},a_{s+c/2+2},\dots,a_{n-1},a_0,a_1,\dots,a_{s+c/2-1}) \]
according as $n/c$ is even or odd.
\end{itemize}
\end{lemma}

\begin{proof}
The fact $c$ is even follows immediately from
    Lemma~\ref{lemma:nega_ciruit_means_nega_tuple_odd_k_odd_c}, and since $c\mid n$ it follows that
    $n$ is also even. By Lemma~\ref{lemma:nega_for_c}, $[a_0,a_1,\dots,a_{c-1}]$ is a
    negasymmetric necklace in $B_k(c-1)$ which contains no negasymmetric $c$-tuples. Since
    $[a_0,a_1,\dots,a_{c-1}]$ is negasymmetric, by definition $a_{i}=-a_{\overline{t-i}}$ for
    some $t$ ($0\leq t\leq c-1$) and every $i$ ($0\leq i\leq c-1$), where the bar indicates
    computation modulo $c$.

By Remark~\ref{remark:n_even_counterexample}, we know $t$ is even, or else there would exist a
    negasymmetric $c$-tuple in the necklace. Hence $a_{t/2}=-a_{t/2}$, i.e.\ $a_{t/2}=0$ or $k/2$.  If we
    set $s=t/2\bmod c$, then $a_{s}=a_{t/2}=0$ or $k/2$, since $[a_0,a_1,\dots,a_{n-1}]$ has period $c$.
    Moreover $0\leq s<c/2$ by definition of $s$, establishing part of (i).

It also follows immediately that $a_{s+i}=-a_{\overline{s-i}}$ for every $i$ ($0\leq i\leq
    c-1$), where $\overline{x}=x\bmod c$.  Thus
    $(a_{s+1},a_{s+2},\dots,a_{c-1},a_0,a_1,\dots,a_{s-1})$ is a negasymmetric $(c-1)$-tuple,
    establishing the first part of (ii).

But by Lemma~\ref{lemma:nega_for_c}(i), this means that $a_{s+c/2}=0$ or $k/2$, completing the
proof of (i). By Lemma~\ref{lemma:nega_moved_by_halfn_is_nega}, it follows that
\[ (a_{s+c/2+1},a_{s+c/2+2},\dots,a_{c-1},a_0,a_1,\dots,a_{s+c/2-1}) \]
is negasymmetric and (ii) follows.

Suppose the first part of (iii) is not true, and without loss of generality relabelling subscripts
as necessary, suppose
\[ (a_0,a_1,\dots,a_{c-2}) \text{~~and~~}
(a_{\overline{t}},a_{\overline{t+1}},\dots,a_{\overline{t-2}}) \]
are both negasymmetric
$(c-1)$-tuples where $0<t<c/2$ and $\overline{x}=x\bmod c$.  Thus
\[ a_i=-a_{c-2-i} \text{~~and~~} a_{\overline{t+i}}=-a_{\overline{t-2-i}} \]
for every $i$ ($0\leq i<c-1$).

Also, if $(a_0,a_1,\dots,a_{c-2})$ is negasymmetric then it must have pseudoweight
 $(c-1)k/2$, and thus $a_{c-1}$ has pseudoweight
$k/2$, and hence $a_{c-1}=0$ or $k/2$. Similarly we have $a_{t-1}=0$ or $k/2$. Hence the equations
above also hold for $i=c-1$.

Setting $i=t+j$ in the first equation gives $a_{\overline{t+j}}=-a_{\overline{-2-t-j}}$ for every
$j$ ($0\leq j<c$), and thus $a_{\overline{-t-2-i}}=a_{\overline{t-2-i}}$ for every $i$ ($0\leq
i<c$). Hence $[a_0,a_1,\dots,a_{c-1}]$ has period dividing $2t<c$, which is a contradiction.

The second part of (iii) follows because if the two $(c-1)$-tuples are equal then
$a_{\overline{s+i}}=a_{\overline{s+c/2+i}}$ for every $i$ ($1\leq i\leq c-1$), where the bar
denotes reduction modulo $c$. Hence, setting $i=c/2$, we have
$a_{\overline{s+c/2}}=a_{\overline{s+c}}=a_{\overline{s}}$, and so
$a_{\overline{s+i}}=a_{\overline{s+c/2+i}}$ for every $i$ ($0\leq i\leq c-1$). Thus
$[a_0,a_1,\dots,a_{n-1}]$ has period dividing $c/2$, giving the desired contradiction.

Now consider (iv). That these $(n-1)$-tuples are negasymmetric follows immediately, because
$(a_0,a_1,\dots,a_{n-1})$ is a concatenation of $n/c$ copies of $(a_0,a_1,\dots,a_{c-1})$.
Moreover, another negasymmetric $(n-1)$-tuple contained in $[a_0,a_1,\dots,a_{n-1}]$ would imply
another negasymmetric $(c-1)$-tuple contained in $[a_0,a_1,\dots,a_{c-1}]$. That they
are distinct follows from exactly the same argument used to show the $(c-1)$-tuples of (ii) are
distinct, except this time the subscripts are computed modulo $n$.

The statement regarding
\[ (a_{s+n/2+1},a_{s+n/2+2},\dots,a_{n-1},a_0,a_1,\dots,a_{s+n/2-1}) \]
is immediate by the periodicity of $[a_0,a_1,\dots,a_{n-1}]$.
\end{proof}

\begin{remark}
An example of the case where a negasymmetric necklace does not contain any negasymmetric $n$-tuples
for $n=4$ and $k=3$ is given in Remark~\ref{remark:n_even_counterexample}. Examples of cases where
a negasymmetric necklace contains two negasymmetric $n$-tuples are provided by
Examples~\ref{example:n4k3} and \ref{example:n6k3}.
\end{remark}

In addition we need the following simple result covering the case where a negasymmetric necklace
$[a_0,a_1,\dots,a_{n-1}]$ contains precisely one negasymmetric $n$-tuple

\begin{lemma}  \label{lemma:exactly_one_nega_n-1-tuple_when_odd_period}
Suppose $k\geq 3$, $n\geq 4$ is even, and $[a_0,a_1,\dots,a_{n-1}]$ is a negasymmetric necklace in
$\mathcal{C}_k(n-1)$ of odd period.  Then $[a_0,a_1,\dots,a_{n-1}]$ contains precisely one
negasymmetric $(n-1)$-tuple.
\end{lemma}

\begin{proof}
Since $[a_0,a_1,\dots,a_{n-1}]$ has odd period, $m$ say, then $[a_0,a_1,\dots,a_{m-1}]$ is a
negasymmetric necklace in $\mathcal{C}_k(m-1)$ which must contain a (unique) negasymmetric
$m$-tuple by Lemma~\ref{lemma:nega_ciruit_means_nega_tuple_odd_k_odd_c}. Suppose this negasymmetric
$m$-tuple is $(d_0,d_1,\dots,d_{m-1})$.  If $(b_0,b_1,\dots,b_{n-1})$ consists of $n/m$
concatenated copies of $(d_0,d_1,\dots,d_{m-1})$, then clearly $(b_0,b_1,\dots,b_{n-1})$ is a
negasymmetric $n$-tuple, and hence $b_i=k-b_{m-1-1}$ for every $i$, $0\leq i<n$. It follows
immediately that $(b_{(m+1)/2},b_{(m+3)/2},\dots,b_{n-1},b_0,b_1,\dots,b_{(m-3)/2})$ is a
negasymmetric $(n-1)$-tuple. Hence we need only show that this is unique.

Suppose without loss of generality that $(a_0,a_1,\dots,a_{n-2})$ and
$(a_{\overline{t}},a_{\overline{t+1}},\dots,a_{\overline{t-2}})$ are both negasymmetric
$(n-1)$-tuples, for some $t$ ($1\leq t<m$), where the bar denotes reduction modulo $n$.  Then
\[ a_{\overline{i}}=-a_{\overline{-2-i}}, \text{~~and~~} a_{\overline{t+i}}=-a_{\overline{t-2-i}}
\]
for every $i$, $0\leq i<n-1$.  Setting $j=t+i$ in the second equation gives
$a_{\overline{j}}=-a_{\overline{2t-2-j}}$ for every $j$, $0\leq j<n~(j\neq t-1)$. Combining with
the first equation this yields $a_{\overline{-2-i}}=-a_{\overline{2t-2-i}}$ for every $i$, $0\leq
i<n~(i\neq n-1 \text{~and~} i\neq t-1)$, and setting $j=-2-i$ means that
\[ a_{\overline{j}}=-a_{\overline{2t+j}} \]
for every $j$, $0\leq j<n-1~(j\neq n-1 \text{~and~} j\neq n-1-t)$.


But, observing that $t<m$ and $m\leq n/2$, this means that
\[ a_{\overline{j}}=-a_{\overline{2t+j}} \]
for every $j$, $0\leq j<m$.  Hence $m\mid 2t$, i.e.\ $m\mid t$ since $m$ is odd.  So
\[ (a_0,a_1,\dots,a_{n-2})=(a_{\overline{t}},a_{\overline{t+1}},\dots,a_{\overline{t-2}}) \]
and the desired result follows.
\end{proof}

We also have a partial converse to Lemma~\ref{lemma:nega_circuit_no_nega_tuples}, as follows.

\begin{lemma}  \label{lemma:converse_to_nega_circuit_no_nega_tuples}
Suppose $k\geq 3$ and $n\geq 4$. Let $n=2^tc$, where $t>0$ and $c$ is odd. If
$(a_{0},a_{1},\dots,a_{n-2})$ is a negasymmetric $(n-1)$-tuple, then both
$[a_0,a_1,\dots,a_{n-2},0]$ and $[a_0,a_1,\dots,a_{n-2},k/2]$ (if $k$ is even) are negasymmetric
necklaces in $\mathcal{C}_k(n-1)$. Moreover, if $[a_0,a_1,\dots,a_{n-2},x]$ contains a
negasymmetric $n$-tuple, where $x$ is either $0$ or $k/2$, then:
\begin{itemize}
\item[(i)] $[a_0,a_1,\dots,a_{n-2},x]$ has odd period dividing $c$;
\item[(ii)] $[a_0,a_1,\dots,a_{n-2},x]$ contains precisely one negasymmetric $n$-tuple;
\item[(iii)] $a_{c-1}=x$, $(a_0,a_1,\dots,a_{n-1})$ equals the concatenation of $2^t$ copies of
    $(a_0,a_1,\dots,a_{c-1})$, and $[a_0,a_1,\dots,a_{c-1}]$ is a negasymmetric necklace
    containing a single negasymmetric $c$-tuple.
\end{itemize}
\end{lemma}

\begin{proof}
Since $(a_{0},a_{1},\dots,a_{n-2})$ is negasymmetric we have
\[ a_i=-a_{n-2-i} \]
for every $i$ ($0\leq i\leq n-2$).  It then follows immediately that, if $x=0$ or $k/2$,
\[ (a_{0},a_{1},\dots,a_{n-2},x)=(-a_{n-2},-a_{n-3}, \dots, -a_0, -x) \]
and so $[a_0,a_1,\dots,a_{n-2},x]$ is a negasymmetric necklace, as required.

Now suppose $[a_0,a_1,\dots,a_{n-2},x]$ contains a negasymmetric $n$-tuple where $x=0$ or $k/2$;
for notational convenience let $a_{n-1}=x$. That is, suppose
\[ (a_{\overline{t}},a_{\overline{t+1}},\dots,a_{\overline{t-1}}) =
 -(a_{\overline{t-1}},a_{\overline{t-2}},\dots,a_{\overline{t}}) \]
for some $t$, where $\overline{y}=y\bmod n$.  That is
\[ a_{\overline{t+i}} = -a_{\overline{t-1-i}} \]
for every $i$ ($0\leq i<n$).  Thus, setting $j=t+i$, we have:
\[ a_{j}=-a_{\overline{2t-1-j}} \]
for every $j$ ($0\leq j<n$).

But we know that $a_i=-a_{n-2-i}$ for every $i$, and hence
\[ a_{\overline{-2-j}} = a_{\overline{2t-1-j}} \]
for all $j$ ($0\leq j<n$) and, setting $i=-2-j$ we get
\[ a_{\overline{i}}=a_{\overline{2t+1+i}} \]
for every $i$ ($0\leq i<n$).  Hence $[a_0,a_{1},..,a_{n-1}]$ has period dividing $2t+1$, i.e.\ it
has odd period, necessarily dividing $c$, and (i) follows.  Result (ii) follows immediately from
Lemma~\ref{lemma:nega_ciruit_means_nega_tuple_odd_k_odd_c} given that $[a_0,a_{1},..,a_{n-1}]$ has
odd period.

For (iii), the fact that $a_{c-1}=x$ and $(a_0,a_1,\dots,a_{n-1})$ equals the concatenation of
$2^t$ copies of $(a_0,a_1,\dots,a_{c-1})$ follows immediately from (i), i.e.\ from the fact that
$[a_0,a_1,\dots,a_{n-2},a_{n-1}]$ has odd period dividing $c$.  Finally,
Lemma~\ref{lemma:nega_for_c} means that $[a_0,a_1,\dots,a_{c-1}]$ is a negasymmetric necklace
containing a single negasymmetric $c$-tuple.
\end{proof}

\begin{example}
If $k=3$, then $(12012)$ is a negasymmetric 5-tuple.  From the above $[120120]$ is a negasymmetric
necklace. It has period 3 and contains the negasymmetric 6-tuple $(201201)$.
\end{example}

Using the results above we can now give the cardinalities of $N_i(n,k)$ for every $i$.

\begin{lemma}\label{lemma:secondformulae N0N1N2_rev}
Suppose $k\geq3$ and $n\geq4$.  Let $n=2^tm$, where $t>0$ and $m$ is odd.  Let $\delta=1$ or $2$
depending whether $k$ is odd or even. Then:
\begin{itemize}
\item[(i)]
 \[ |N_0(n,k)| =
  \frac{\delta}{2} (\delta k^{(n-2)/2}-k^{(m-1)/2}); \]
\item[(ii)] \[ |N_1(n,k)| = \delta k^{(m-1)/2},\] and every necklace in $N_1(n,k)$ has  period
    dividing $m$;
\item[(iii)] \[ |N_2(n,k)| = \frac{k^{n/2}-\delta k^{(m-1)/2}}{2}. \]
\end{itemize}
\end{lemma}

\begin{proof}
\begin{itemize}
\item[(i)]

It follows from Lemma~\ref{lemma:converse_to_nega_circuit_no_nega_tuples} that every
negasymmetric $(n-1)$-tuple $(a_0,a_1,\dots,a_{n-2})$ is contained in $\delta$ distinct
negasymmetric necklaces $[a_0,a_1,\dots,a_{n-1}]$ in $\mathcal{C}_k(n-1)$, each of which
contains either zero or one negasymmetric $n$-tuple(s) and has even or odd period,
respectively. There are $\delta k^{(n-2)/2}$ such $(n-1)$-tuples, from
Lemma~\ref{lemma:number_of_nega_tuples}

From Lemma~\ref{lemma:nega_circuit_no_nega_tuples}, a negasymmetric necklace
$[a_0,a_1,\dots,a_{n-1}]$ in $\mathcal{C}_k(n-1)$ of even period containing no negasymmetric
$n$-tuples, of which there are $|N_0(n,k)|$, contains precisely 2 distinct negasymmetric
$(n-1)$-tuples.

From Lemma~\ref{lemma:exactly_one_nega_n-1-tuple_when_odd_period}, a negasymmetric necklace
$[a_0,a_1,\dots,a_{n-1}]$ in $\mathcal{C}_k(n-1)$ of odd period containing one negasymmetric
$n$-tuple, of which there are $|N_1(n,k)|$, contains a unique negasymmetric $(n-1)$-tuple.

If we count (negasymmetric $(n-1)$-tuple, necklace) pairs in two ways, where the $(n-1)$-tuple
is in the necklace, we get
\[ \delta^2 k^{(n-2)/2} = 2|N_0(n,k)| + |N_1(n,k)|. \]
It follows immediately (observing that $|N_1(n,k)|=\delta k^{(m-1)/2}$ from part (ii)) that
\[ |N_0(n,k)| = \frac{\delta}{2}(\delta k^{(n-2)/2}-k^{(m-1)/2}).  \]

\item[(ii)] From Lemma~\ref{lemma:nega_ciruit_means_nega_tuple_odd_k_odd_c}, a negasymmetric
    necklace contains a single negasymmetric $n$-tuple if and only if it has odd period.  That
    is, a negasymmetric necklace in $N_1(n,k)$ must have period dividing $m$. By
    Lemma~\ref{lemma:number_of_nega_tuples}, the number of negasymmetric $m$-tuples for odd $m$
    is $\delta k^{(m-1)/2}$, where $\delta$ is as in the statement of the theorem, and so
    $|N_1(n,k)| = \delta k^{(m-1)/2}$.
\item[(iii)] By Lemma~\ref{lemma:number_of_nega_tuples}, there are $k^{n/2}$ negasymmetric
    $n$-tuples, since $n$ is even.  Of these, $N_1(n,k)$ appear in a negasymmetric necklace of
    odd period, one $n$-tuple per necklace. The remaining $k^{n/2}-N_1(n,k)$ negasymmetric
    $n$-tuples appear in the $|N_2(n,k)|$ negasymmetric necklaces containing two negasymmetric
    $n$-tuples each.  Hence $2|N_2(n,k)|+|N_1(n,k)|=k^{n/2}$, and the result follows.
\end{itemize}
\end{proof}

Using Lemmas~\ref{lemma:N1_for_n_odd} and \ref{lemma:secondformulae N0N1N2_rev},
the values of $|N_i(n,k)|$ are tabulated for small values of $n$ and
$k$ in Table~\ref{table:N_i(n,k)}, where the values are given as a triple for $i=0,1,2$.

\begin{table}[htb]
\centering \caption{Values of $|N_i(n,k)|$ for small $n$ and $k$} \label{table:N_i(n,k)}
\begin{tabular}{crrrr} \hline
$n$& $k=3$  & $k=4$    & $k=5$    & $k=6$     \\ \hline
3  & (0,3,0)   & (0,8,0)     & (0,5,0)     & (0,12,0)      \\ \hline
4  & (1,1,4)   & (7,2,7)     & (2,1,12)     & (11,2,17)      \\ \hline
5  & (0,9,0)   & (0,32,0)     & (0,25,0)     & (0,72,0)      \\ \hline
6  & (3,3,12)   & (28,8,28)     & (10,5,60)     & (66,12,102)      \\ \hline
7  & (0,27,0)   & (0,128,0)     & (0,125,0)     & (0,432,0)      \\ \hline
8  & (13,1,40)   & (127,2,127)     & (62,1,312)     & (431,2,647)      \\ \hline
\end{tabular}
\end{table}

\subsection{The main result}

We can now combine Lemmas~\ref{lemma:N1_for_n_odd} and \ref{lemma:secondformulae N0N1N2_rev} to
give an upper bound on the size of $X_k(n-1)$.

\begin{theorem}  \label{theorem:X_bound}
Suppose $k\geq3$ and $n\geq 3$. Let $n=2^tm$ where $t>0$ and $m$ is odd. Let $\delta=1$ if $k$ is
odd and $\delta=2$ if $k$ is even. Suppose $X_k(n-1)$ is as defined in
Theorem~\ref{theorem:Xkn-1_is_Eulerian}. Then:
\[ |X_k(n-1)|\geq \max\left(\frac{k^n-r_{k,n,kn/2}}{2},~s_k(n-1)\right) \]
where
 \[ s_k(n-1) = \frac{k^n - n(|N_0(n,k)|+|N_2(n,k)|)-m(|N_1(n,k)|-\delta)-\delta}{2}
 \]
or, equivalently,
\[ s_k(n-1) = \left\{
\begin{aligned}
\frac{k^n-\delta(n(k^{(n-1)/2}-1)-1)}{2} \text{~~~~if $n$ is odd~} \\
\frac{k^n - \frac{nk^{(n-2)/2}}{2}(\delta^2+k) +n\delta k^{(m-1)/2} - m\delta(k^{(m-1)/2}-1)
-\delta}{2} \\ \text{~~~~~~~~~~~~~~~~~~if $n$ is even~}
\end{aligned} \right..  \]
\end{theorem}

\begin{proof}
By definition, $|X_k(n-1)|=|E_k(n-1)|+w/2$, where $w$ is the total number of edges in the necklaces
in $X_k(n-1)$ taken from $\mathcal{C}_k(n-1)$.  By Lemma~\ref{lemma:bound on w n evenx}, $w$ is at
least
\[ r_{k,n,kn/2} - n(|N_0(n,k)|+|N_2(n,k)|)-m(|N_1(n,k)|-\delta)-\delta. \]
As noted in Section~\ref{section:introduction}, $|E_k(n-1)|=(k^n-r_{k,n,kn/2})/2$, and the first
inequality follows.

The second inequality for the $n$ odd case follows immediately from Lemma~\ref{lemma:N1_for_n_odd},
observing that in this case $m=n$ and $|N_0(n,k)|=|N_2(n,k)|=0$. If $n$ is even, from
Lemma~\ref{lemma:secondformulae N0N1N2_rev},
\begin{equation*}
\begin{split}
|N_0(n,k)|+|N_2(n,k)| & = \frac{\delta}{2} (\delta k^{(n-2)/2}-k^{(m-1)/2})
+\frac{k^{n/2}-\delta k^{(m-1)/2}}{2} \\
& = \frac{k^{(n-2)/2}}{2}(\delta^2+k)-\delta k^{(m-1)/2},
\end{split}
\end{equation*}
and
\[ |N_1(n,k)|= \delta k^{(m-1)/2}, \]
and the result now follows.
\end{proof}

The bounds of Theorem~\ref{theorem:X_bound} are tabulated for small $k$ and $n$ in
Table~\ref{table:X_bound}. More specifically, the values of $s_k(n-1)$ are given in the table, with
the corresponding value of $\frac{k^n-r_{k,n,kn/2}}{2}$ given in brackets.  That is, the actual
lower bound is the larger of the two values.

\begin{table}[htb]
\centering \caption{Lower bounds on the size of $X_k(n-1)$} \label{table:X_bound}
\begin{tiny}
\begin{tabular}{crrrrrrrr} \hline
$n$& $k=3$& $k=4$  & $k=5$   & $k=6$    & $k=7$    & $k=8$    & $k=9$     & $k=10$    \\ \hline
3  & 10   & 22     & 56      & 92       & 162      & 234      & 352       & 472       \\
   & (10) & (22)   & (56)    & (89)     & (162)    & (225)    & (352)     & (454)     \\ \hline
4  & 30   & 99     & 284     & 591      & 1146     & 1955     & 3192      & 4863      \\
   & (31) & (93)   & (278)   & (550)    & (1109)   & (1835)   & (3084)    & (4604)    \\ \hline
5  & 101  & 436    & 1502    & 3712     & 8283     & 16068    & 29324     & 49504     \\
   & (96) & (386)  & (1432)  & (3362)   & (8008)   & (14858)  & (28624)   & (46449)   \\ \hline
6  & 316  & 1870   & 7596    & 22808    & 58248    & 129946   & 264520    & 497932    \\
   & (294)& (1586) & (7162)  & (20441)  & (55518)  & (119895) & (254004)  & (467468)  \\ \hline
7  & 1002 & 7750   & 38628   & 138462   & 410574   & 1044998  & 2388936   & 4993006   \\
   & (897)& (6476) & (36220) & (123895) & (393991) & (965569) & (2321848) & (4697914) \\ \hline
8  & 3068 & 31751  & 193816  & 835495   & 2876916  & 8376327  & 21508784  & 49972007  \\
   &(2727)& (26333)& (181550)& (749422) & (2748581)& (7766075)& (20750748)& (47167644)\\ \hline
9  & 9481 & 128776 & 973754  & 5027192  & 20166003 & 67072008 & 193680724 & 499910008 \\
   &(8272)&(106762)& (912944)& (4526720)&(19373760)&(62405190)&(188369056)&(473247274)\\ \hline
\end{tabular}
\end{tiny}
\end{table}

Observe that the bound of Theorem~\ref{theorem:X_bound} is only a lower bound on the size of
$|X_k(n-1)|$, and the bound is only tight when $n$ is prime. This is exemplified by the analyses in
Examples~\ref{example:n6k3} and \ref{example:n4k4} that yield sets of $n$-tuples slightly larger
than the value given in Table~\ref{table:OS_Xkn-1_periods}.

\section{Constructing negative orientable and orientable sequences}  \label{section:NOSs} \label{section:OSs}

\subsection{Negative orientable sequences}

Theorem~\ref{theorem:Xkn-1_is_Eulerian} states that $X_k(n-1)$ is an Eulerian subgraph of
$B_k(n-1)$ and is antinegasymmetric.  Following the same argument as employed in \cite[Section
3.2]{Mitchell25a}, this immediately means that a directed circuit in $X_k(n-1)$, which must exist
because it is Eulerian, will yield a $\mathcal{NOS}_k(n)$ of period $|X_k(n-1)|$.  Hence it is
possible to construct a $\mathcal{NOS}_k(n)$ with period at least that given by
Theorem~\ref{theorem:X_bound}
--- see also Table~\ref{table:X_bound}. Previously, the sequences with the largest known periods
were derived from directed circuits in $E_k({n-1})\subseteq X_k(n-1)$, i.e.\ they had periods as
given by the bracketed values in Table~\ref{table:X_bound}.

In fact, as we claimed in the abstract, these sequences are asymptotically best possible with
regards to period.  Let $\ell(n,k)$ denote the maximum period for a negative orientable sequence of
order $n$ over an alphabet of size $k$.

We first need the following trivial result.

\begin{lemma}[Corollary 3.2 of \cite{Mitchell25a}]  \label{lemma:simple_NOS_bound}
Suppose $n\geq2$ and $k\geq2$. Then
\[
\ell(n,k) \leq \left\{
\begin{aligned}
\frac{k^n-\delta k^{(n-1)/2}}{2} \text{~~~~if $n$ is odd} \\
\frac{k^n - k^{n/2}}{2}\text{~~~~if $n$ is even}
\end{aligned} \right.
\]
where, as previously, $\delta=1$ if $k$ is odd and $\delta=2$ if $k$ is even.
\end{lemma}

\begin{remark}
Of course, sharper bounds are known (see, for example, \cite[Theorem 3.5]{Mitchell25a}), but this
simple bound suffices for our purposes here.
\end{remark}

\begin{theorem}
Suppose $n\geq3$ and $k\geq3$. The negative orientable sequences generated using $X_k(n-1)$ have
asymptotically optimal period with respect to both $n$ and $k$.  That is
\[ \frac{s_k(n-1)}{\ell(n,k)}\rightarrow 1 \text{~~as~}n\rightarrow\infty \]
and
\[ \frac{s_k(n-1)}{\ell(n,k)}\rightarrow 1 \text{~~as~}k\rightarrow\infty. \]
\end{theorem}

\begin{proof}
Denote the bound of Lemma~\ref{lemma:simple_NOS_bound} by $b(n,k)$.  First suppose $n$ is odd.
Then:

\begin{align*}
\ell(n,k)-s_k(n-1) & \leq b(n,k)-s_k(n-1) \\
& = \frac{k^n-\delta k^{(n-1)/2}}{2} - \frac{k^n-\delta(n(k^{(n-1)/2}-1)-1)}{2} \\
& = \delta\frac{n(k^{(n-1)/2}-1)-1)-k^{(n-1)/2}}{2} \\
& = \delta\frac{(n-1)k^{(n-1)/2}-n-1}{2} < \delta\frac{nk^{(n-1)/2}}{2}.
\end{align*}

Now, trivially, $\ell(n,k)\geq s_k(n-1)$, and hence:

\begin{align*}
\frac{\ell(n,k)-s_k(n-1)}{s_k(n-1)}& < \frac{\delta\frac{nk^{(n-1)/2}}{2}}{\frac{k^n-\delta(n(k^{(n-1)/2}-1)-1)}{2}} \\
& < \frac{\delta nk^{(n-1)/2}}{k^n-\delta nk^{(n-1)/2})} \\
& \rightarrow 0 \text{~as either~} n\rightarrow\infty \text{~or~} k\rightarrow\infty,
\end{align*}
and the result follows.

Now suppose $n$ is even. First observe that, since $n\geq m$,
\[ s_k(n-1) > \frac{k^n - \frac{nk^{(n-2)/2}}{2}(\delta^2+k)}{2}. \]
Hence
\[ b(n,k)-s_k(n-1) < \frac{\frac{nk^{(n-2)/2}}{2}(\delta^2+k) - k^{n/2}}{2}, \]
and so
\begin{align*}
\frac{\ell(n,k)-s_k(n-1)}{s_k(n-1)} & <
\frac{\frac{nk^{(n-2)/2}}{2}(\delta^2+k) - k^{n/2}}{k^n - \frac{nk^{(n-2)/2}}{2}(\delta^2+k)} \\
< & \frac{\frac{nk^{(n-2)/2}}{2}(\delta^2+k)}{k^n - \frac{nk^{(n-2)/2}}{2}(\delta^2+k)} \\
& \rightarrow 0 \text{~as either~} n\rightarrow\infty \text{~or~} k\rightarrow\infty,
\end{align*}
as required.
\end{proof}

\subsection{Orientable sequences}

We next show how to construct a large period orientable sequence from $X_k(n-1)$; following the
approach of \cite{Mitchell26}, we first define the Lempel homomorphism $D$, mapping from $B_k(n)$
to $B_{k}(n-1)$. We follow Alhakim and Akinwande \cite{Alhakim11}, who generalised the original
Lempel definition \cite{Lempel70} for $k=2$, to alphabets of arbitrary size.

\begin{definition}[Alhakim and Akinwande, \cite{Alhakim11}] \label{definition:Lempel}
Define the function $D_{\beta}:B_k(n)\rightarrow B_{k}(n-1)$ as follows, where
$\beta\in\mathbb{Z}^*_k$. If $\mathbf{a} = (a_0,a_1,\ldots,a_{n-1})\in \mathbb{Z}_k^n$ then
\[ D_{\beta}(\mathbf{a}) = (\beta(a_1-a_0),\beta(a_2-a_1),\ldots\beta(a_{n-1}-a_{n-2})). \]
\end{definition}

We are particularly interested in the case $\beta=1$, and in what follows we simply write $D$ for
$D_1$.  The following result is key.

\begin{lemma}[Lemma 5.2 of \cite{Mitchell26}]  \label{lemma:Lempel_antisymmetry}
Suppose $n\geq2$ and $k\geq3$.  If $T$ is an antinegasymmetric subgraph of the de Bruijn digraph
$B_{k}(n-1)$ with edge set $E$, then $D^{-1}(E)$, of cardinality $k|E|$, is the set of edges for an
antisymmetric subgraph of $B_{k}(n)$, which, abusing our notation slightly, we refer to as
$D^{-1}(T)$. Moreover, if every vertex of $T$ has in-degree equal to its out-degree, then the same
apples to $D^{-1}(T)$.
\end{lemma}

Combining Lemma~\ref{lemma:multiple_negasymmetric_tuples} with
Theorem~\ref{theorem:Xkn-1_is_Eulerian} immediately gives the following.

\begin{corollary}  \label{corollary:D-1Xkn-1_is_nearly_Eulerian}
If $n\geq3$ and $k\geq 3$ then $D^{-1}(X_k(n-1))$ is the set of edges of an antisymmetric subgraph
of $B_{k}(n)$ in which every vertex of $T$ has in-degree equal to its out-degree.
\end{corollary}

We also have the following.

\begin{lemma}  \label{lemma:D-1Xkn-1_is_connected}
If $n\geq3$ and $k\geq 3$ then $D^{-1}(X_k(n-1))$ is the set of edges for a connected subgraph of
$B_k(n)$.
\end{lemma}

\begin{proof}
Theorem 5.5 of \cite{Mitchell26} shows that the set of edges in $D^{-1}(E_k(n-1))$ forms a
connected subgraph of $B_k(n)$.  Thus we need only show that any edge in
$D^{-1}(X_k(n-1)-E_k(n-1))$ belongs to the same connected component of $D^{-1}(X_k(n-1))$ as an
edge in $D^{-1}(E_k(n-1))$.  If $\mathbf{a}=(a_0,a_1,\dots,a_n)$ is such an edge, then
$D(\mathbf{a})$ is in $X_k(n-1)-E_k(n-1))$. But by Theorem~\ref{theorem:Xkn-1_is_Eulerian},
$X_k(n-1)$ is Eulerian, and hence there is a linked sequence of directed edges connecting
$D(\mathbf{a})$ to an edge in $E_k(n-1)$, where each directed edge links to the next edge via a
common vertex. This immediately can be used, by applying $D^{-1}$, to construct a linked sequence
of directed edges linking $\mathbf{a}$ to an edge in $D^{-1}(E_k(n-1))$ and the result follows.
\end{proof}

It follows immediately from Corollary~\ref{corollary:D-1Xkn-1_is_nearly_Eulerian} and
Lemma~\ref{lemma:D-1Xkn-1_is_connected} that any directed circuit in $D^{-1}(X_k(n-1))$ will yield
an $\mathcal{OS}_k(n+1)$. We immediately have the following.

\begin{theorem} \label{theorem:new_OS_lower_bounds}
If $n\geq3$ and $k\geq3$ then there exists an $\mathcal{OS}_k(n+1)$ of period $ks_k(n-1)$, where
$s_k(n-1)$ is as specified in Theorem~\ref{theorem:X_bound}.
\end{theorem}

The lower bounds on the periods of the sequences of Theorem~\ref{theorem:new_OS_lower_bounds} for
small $k$ and $n$ are given in Table~\ref{table:OS_Xkn-1_periods}, along with the periods of the
sequences from \cite[Table 5]{Mitchell26} in brackets for comparison.

\begin{table}[htb]
\centering \caption{Periods of the constructed $\mathcal{OS}_k(n)$ (and bounds)}
\label{table:OS_Xkn-1_periods}

\begin{tabular}{c|rrrrrr} \hline
$n$ & $k=3$   & $k=4$    & $k=5$    & $k=6$    & $k=7$     & $k=8$     \\ \hline
4   & 30      & 88       & 280      & 552      & 1134      & 1872      \\
    & (30)    & (88)     & (280)    & (534)    & (1134)    & (1800)    \\ \hline
5   & 90      & 396      & 1420     & 3546     & 8022      & 15640     \\
    & (93)    & (372)    & (1390)   & (3300)   & (7763)    & (14680)   \\ \hline
6   & 303     & 1744     & 7510     & 22272    & 57981     & 128544    \\
    & (288)   & (1544)   & (7160)   & (20172)  & (56056)   & (118864)  \\ \hline
7   & 948     & 7480     & 37980    & 136848   & 407736    & 1039568   \\
    & (882)   & (6344)   & (35810)  & (122646) & (388626)  &(959160)   \\ \hline
8   & 3006    & 31000    & 193140   & 830772   & 2874018   & 8359984   \\
    & (2691)  & (25904)  & (181100) & (743370) & (2757937) & (7724552) \\ \hline
\end{tabular}

\end{table}

It is important to note that Theorem~\ref{theorem:new_OS_lower_bounds} only gives a lower bound on
the period of the orientable sequences obtained by the approach described in this paper, and the
bound is only tight when $n-1$ is prime.  As previously noted in connection with
Table~\ref{table:X_bound}, this is exemplified by the analyses in Examples~\ref{example:n6k3} and
\ref{example:n4k4} that yield sequences with a slightly larger period than that given in
Table~\ref{table:OS_Xkn-1_periods}.

Finally, as for the negative orientable case, we show that the sequences of
Theorem~\ref{theorem:new_OS_lower_bounds} have asymptotically optimal period.  Let $\lambda(n,k)$
denote the maximum period for an orientable sequence of order $n$ over an alphabet of size $k$. The
following result is elementary.

\begin{lemma}[Lemma 16 of \cite{Burns93}]  \label{lemma:simple_OS_bound}
Suppose $n\geq2$ and $k\geq2$. Then
\[
\lambda(n,k) \leq \left\{
\begin{aligned}
\frac{k^n - k^{(n+1)/2}}{2} \text{~~~~if $n$ is odd} \\
\frac{k^n - k^{n/2}}{2}     \text{~~~~if $n$ is even}
\end{aligned} \right..
\]
\end{lemma}

\begin{remark}
As with Lemma~\ref{lemma:simple_NOS_bound}, sharper bounds are known (see, for example,
\cite[Theorem 2.13]{Mitchell26}), but this simple bound suffices here.
\end{remark}

\begin{theorem}
Suppose $n\geq3$ and $k\geq3$. The orientable sequences of
Theorem~\ref{theorem:new_OS_lower_bounds} have asymptotically optimal period with respect to both
$n$ and $k$.  That is
\[ \frac{ks_k(n-1)}{\lambda(n+1,k)}\rightarrow 1 \text{~~as~}n\rightarrow\infty \]
and
\[ \frac{ks_k(n-1)}{\lambda(n+1,k)}\rightarrow 1 \text{~~as~}k\rightarrow\infty. \]
\end{theorem}

\begin{proof}
Denote the bound of Lemma~\ref{lemma:simple_OS_bound} by $c(n,k)$.  First suppose $n$ is odd. Then:

\begin{align*}
\lambda(n+1,k)-ks_k(n-1) & \leq c(n+1,k)-ks_k(n-1) \\
& = \frac{k^{n+1}-k^{(n+1)/2}}{2} - \frac{k^{n+1}-\delta k(n(k^{(n-1)/2}-1)-1)}{2} \\
& = \frac{\delta k(n(k^{(n-1)/2}-1)-1)-k^{(n+1)/2}}{2} \\
& = \frac{\delta nk^{(n+1)/2}-\delta k(n+1)-k^{(n+1)/2}}{2} < \frac{\delta nk^{(n+1)/2}}{2}.
\end{align*}

Now, trivially, $\lambda(n+1,k)\geq ks_k(n-1)$, and hence:

\begin{align*}
\frac{\lambda(n+1,k)-ks_k(n-1)}{ks_k(n-1)} & < \frac{\delta nk^{(n+1)/2}}{k^{n+1}-\delta k(n(k^{(n-1)/2}-1)-1)} \\
& < \frac{\delta nk^{(n+1)/2}}{k^{n+1}-\delta nk^{(n+1)/2})} \\
& \rightarrow 0 \text{~as either~} n\rightarrow\infty \text{~or~} k\rightarrow\infty,
\end{align*}
and the result follows.

Now suppose $n$ is even. First observe that, since $n\geq m$,
\[ ks_k(n-1) > \frac{k^{n+1} - \frac{nk^{n/2}}{2}(\delta^2+k)}{2}. \]
Hence
\[ c(n+1,k)-ks_k(n-1) < \frac{\frac{nk^{n/2}}{2}(\delta^2+k) - k^{(n+2)/2}}{2}, \]
and so
\begin{align*}
\frac{\lambda(n+1,k)-ks_k(n-1)}{ks_k(n-1)} & <
\frac{\frac{nk^{n/2}}{2}(\delta^2+k) - k^{(n+2)/2}}{k^{n+1} - \frac{nk^{n/2}}{2}(\delta^2+k)}\\
& < \frac{\frac{nk^{n/2}}{2}(\delta^2+k)}{k^{n+1} - \frac{nk^{n/2}}{2}(\delta^2+k)}\\
& \rightarrow 0 \text{~as either~} n\rightarrow\infty \text{~or~} k\rightarrow\infty,
\end{align*}
as required.
\end{proof}

\section{Concluding remarks}  \label{section:conclusions}

It is important to assess how the construction method described here affects efforts to find
orientable sequences with the largest possible period.  The current state of knowledge for small
$n$ and $k>2$ is summarised in Table~\ref{table:OS_period_bounds}, where the upper bound from
\cite[Theorem 2.13]{Mitchell26} is given in brackets beneath the largest known period.

\begin{table}[htb]
\centering \caption{Largest known periods for an $\mathcal{OS}_k(n)$ (and bounds)}
\label{table:OS_period_bounds}

\begin{tabular}{crrrrrr} \hline
$n$ & $k=3$   & $k=4$    & $k=5$    & $k=6$    & $k=7$     & $k=8$     \\ \hline
2   & {\bf 3} & {\bf 4}  & {\bf 10} & {\bf 12} & {\bf 21}  & {\bf 24}  \\
    & (3)     & (4)      & (10)     & (12)     & (21)      & (24)      \\ \hline
3   & {\bf 9} & {\bf 20} & {\bf 50} & {\bf 84} & {\bf 147} & {\bf  216}\\
    & (9)     & (20)     & (50)     & (84)     & (147)     & (216)     \\ \hline
4   & {\bf 30}& 88       & {\bf 280}& 552      & {\bf 1134}&  1872     \\
    & (30)    & (112)    & (280)    & (612)    & (1134)    & (1984)    \\ \hline
5   & 93      & 404      & 1420     & 3546     & 8022      &  15640    \\
    & (99)    & (452)    & (1450)   & (3684)   & (8085)    & (15896)   \\ \hline
6   & 303     & 1744     & 7510     & 22272    & 57981     &  128544   \\
    & (315)   & (1958)   & (7550)   & (23019)  & (58065)   & (130332)  \\ \hline
7   & 954     & 7480     & 37980    & 136848   & 407736    &  1039568  \\
    & (972)   & (7844)   & (38100)  & (138144) & (408072)  & (1042712) \\ \hline
8   & 3006    & 31000    & 193140   & 830772   & 2874018   &  8359984  \\
    & (3096)  & (32390)  & (193800) & (837879) & (2876496) & (8382492) \\ \hline
\end{tabular}
\end{table}

The following observations can be made about this table.  The values in bold represent maximal
values.
\begin{itemize}
\item $n=2$:  The fact that there exists an $\mathcal{OS}_k(2)$ with period meeting the bound
    follows from \cite[Theorem 5.4]{Alhakim24a} (for $k$ prime), and from \cite[Theorem
    2]{Gabric25} and \cite[Lemma 2.2]{Mitchell25a} for general $k$.
\item $n=3$: The existence of an $\mathcal{OS}_k(3)$ meeting the period bound is due to
    \cite[Example 5.1]{Alhakim24a} for $k=3$, \cite[Section 1]{Gabric25} for $k=5$ (from an
    exhaustive search), and for general $k$ from \cite{Mitchell26}.
\item $n=4$: The fact that the maximum period of an $\mathcal{OS}_3(4)$ is 30 is again due to
    \cite[Section 1]{Gabric25}, and the existence of an $\mathcal{OS}_k(4)$ meeting the bound
    for general odd $k$ is from \cite{Mitchell26}. the values for $k=6$ and $k=8$ follow from
    this paper.
\item $n=5$: The $\mathcal{OS}_3(5)$ of period 93 comes from \cite{Mitchell26}. The values for
    $k>3$ come from this paper. The fact that the value for $n=5$ and $k=4$ is 404, and not 396
    as guaranteed by the bound of Theorem~\ref{theorem:new_OS_lower_bounds}, arises from the
    analysis in Example~\ref{example:n4k4}.
\item $n\geq6$: All values in the table come from this paper. Analogously to the case $n=5$ and
    $k=4$, the value of 954, and not 948, for $n=7$ and $k=3$ arises from the analysis in
    Example~\ref{example:n6k3}.
\end{itemize}

That is, the results in this paper appear to yield sequences with larger period than were
previously known for all values of $n\geq5$. In particular, the values in the table
exceed those obtained by Gabri\'{c} and Sawada \cite{Gabric25}, whose construction technique also
yields sequences with asymptotically optimal period. As has been observed above,
except when $n-1$ is prime, the value yielded by the bound of
Theorem~\ref{theorem:new_OS_lower_bounds} is typically less than the value actually obtained.

However, there are many questions that remain to be resolved.  We briefly mention two here.
\begin{itemize}
\item First, and as noted in Remark~\ref{remark:other_partitions}, we have focussed in this
    paper on adding edges from non-negasymmetric necklaces in $\mathcal{C}_k(n-1)$ to
    $E_k(n-1)$. However, if $k$ is even then $\mathcal{C}_k(n-1)$ is by no means the only way
    of dividing the $n$-tuples of pseudoweight $kn/2$ into circuits, and there may be other
    ways of doing this that enable a larger set of $n$-tuples to be added to $E_k(n-1)$ while
    preserving antinegasymmetry and the Eulerian property.
\item Second, whilst the approach described in this paper results in orientable and negative
    orientable sequences with period greater than any previously known sequences, it is far
    from clear whether or not sequences with greater period can be constructed.
\end{itemize}
That is, while progress has been achieved, the search for optimal period orientable sequences is
not over yet!


\begin{thebibliography}{10}

\bibitem{Alhakim11} A.~Alhakim and M.~Akinwande, \emph{A recursive construction of nonbinary de
  {Bruijn} sequences}, Des. Codes Cryptogr. \textbf{60} (2011), no.~2,
  155--169.

\bibitem{Alhakim24a} A.~Alhakim, C.~J. Mitchell, J.~Szmidt, and P.~R. Wild, \emph{Orientable
  sequences over non-binary alphabets}, Cryptogr. Commun. \textbf{16} (2024),
  1309--1326.

\bibitem{Burns92} J.~Burns and C.~J. Mitchell, \emph{Coding schemes for two-dimensional
    position
  sensing}, Tech. Report HPL--92--19, January 1992,
  \url{https://www.chrismitchell.net/HPL-92-19.pdf}.

\bibitem{Burns93} \bysame, \emph{Coding schemes for two-dimensional position sensing},
  Cryptography and Coding III (M.~J. Ganley, ed.), Oxford University Press,
  1993, pp.~31--66.

\bibitem{cameron25} B.~Cameron, A.~G{\"u}ndo{\u g}an, and J.~Sawada, \emph{Cut-down de {Bruijn}
  sequences}, Discrete Math. \textbf{348} (2025), 114024.

\bibitem{Dai93} Z.-D. Dai, K.~M. Martin, M.~J.~B. Robshaw, and P.~R. Wild, \emph{Orientable
  sequences}, Cryptography and Coding III (M.~J. Ganley, ed.), Oxford
  University Press, Oxford, 1993, pp.~97--115.

\bibitem{Gabric24} D.~Gabri{\'c} and J.~Sawada, \emph{Construction of orientable sequences in
  {$O(1)$-amortized} time per bit}, 2024, Available at
  \url{https://arxiv.org/abs/2401.14341}.

\bibitem{Gabric24b} \bysame, \emph{Efficient construction of long orientable sequences}, 35th
  Annual Symposium on Combinatorial Pattern Matching, {CPM} 2024, June 25--27,
  2024, Fukuoka, Japan (S.~Inenaga and S.~J. Puglisi, eds.), LIPIcs, vol. 296,
  Schloss Dagstuhl --- Leibniz-Zentrum f{\"{u}}r Informatik, 2024,
  pp.~15:1--15:12.

\bibitem{Gabric25} \bysame, \emph{Constructing $k$-ary orientable sequences with asymptotically
  optimal length}, Des. Codes Cryptogr. \textbf{93} (2025), 2349--2367.

\bibitem{Lempel70} A.~Lempel, \emph{On a homomorphism of the de {B}ruijn graph and its
    application
  to the design of feedback shift registers}, {IEEE} Trans. Comput.
  \textbf{{\bf C-19}} (1970), 1204--1209.

\bibitem{Mitchell22} C.~J. Mitchell and P.~R. Wild, \emph{Constructing orientable sequences},
    IEEE
  Trans. Inform. Theory \textbf{68} (2022), 4782--4789.

\bibitem{Mitchell25a} \bysame, \emph{Orientable and negative orientable sequences}, Discrete
    Appl.
  Math. \textbf{377} (2025), 242--259.

\bibitem{Mitchell26} \bysame, \emph{New orientable sequences}, Acta Inform. \textbf{63} (2026),
    14.

\bibitem{Ruskey92} F.~Ruskey, C.~D. Savage, and T.~M.~Y. Wang, \emph{Generating necklaces}, J.
  Algorithms \textbf{13} (1992), no.~3, 414--430.

\end{thebibliography}

\providecommand{\bysame}{\leavevmode\hbox to3em{\hrulefill}\thinspace}
\providecommand{\MR}{\relax\ifhmode\unskip\space\fi MR }
\providecommand{\MRhref}[2]{%
  \href{http://www.ams.org/mathscinet-getitem?mr=#1}{#2}
} \providecommand{\href}[2]{#2}

\end{document}